\documentclass[12pt]{amsart}
\usepackage{amssymb}

\usepackage{stmaryrd}
\usepackage{mathabx}
\usepackage{mathptmx}
\usepackage{cleveref}
\usepackage[all]{xy}
\usepackage{color}
\usepackage{tikz}
\usepackage{todonotes}
\usepackage{paralist}
\numberwithin{equation}{section}

\newcommand{\mm}{\mathfrak m}
\newcommand{\Z}{\mathbb{Z}}
\newcommand{\K}{\mathbb{K}}
\newcommand{\RR}{\mathbb{R}}

\newcommand{\N}{\mathbb{N}}

\newcommand{\eb}{{\bf e}}

\newcommand{\mb}{{\bf m}}

\DeclareMathOperator{\chara}{char}

\DeclareMathOperator{\cs}{c-sqfree}
\DeclareMathOperator{\cSupp}{cSupp}

\DeclareMathOperator{\gin}{gin}

\DeclareMathOperator{\Hilb}{Hilb}

\DeclareMathOperator{\ini}{in}

\DeclareMathOperator{\lk}{lk}
\DeclareMathOperator{\Mon}{Mon}

\DeclareMathOperator{\sm}{sm}

\DeclareMathOperator{\Tor}{Tor}

\DeclareMathOperator{\pnt}{\raise 0.5mm \hbox{\large\bf.}}

\newcommand{\bal}{\mathrm{bal}}
\newcommand{\mf}{\mathfrak{m}}

\let\phi=\varphi

\newtheorem{theorem}{Theorem}[section]
\newtheorem{lemma}[theorem]{Lemma}
\newtheorem{proposition}[theorem]{Proposition}
\newtheorem{corollary}[theorem]{Corollary}

\newtheorem{lem-def}[theorem]{Lemma and Definition}
\newtheorem{prop-def}[theorem]{Proposition and Definition}

\theoremstyle{definition}
\newtheorem{definition}[theorem]{Definition} % \theoremstyle{remark}
\newtheorem{remark}[theorem]{Remark}
\newtheorem{rem-def}[theorem]{Remark and Definition}
\newtheorem{example}[theorem]{Example}

\title{Balanced squeezed Complexes}

\author{Martina Juhnke-Kubitzke}
\address{Universit\"at Osnabr\"uck, Institut f\"ur Mathematik, 49069 Osnabr\"uck, Germany}
\email{juhnke-kubitzke@uos.de}

\author[Uwe Nagel]{Uwe Nagel}
\address{Department of Mathematics, University of Kentucky, 715 Patterson Office Tower, Lexington, KY 40506-0027, USA}
\email{uwe.nagel@uky.edu}

\begin{document}

\begin{abstract} 
Given any order ideal $U$ consisting of color-squarefree monomials involving variables with $d$ colors, we associate to it a balanced $(d-1)$-dimensional  simplicial complex $\Delta_{\bal}(U)$ that we call a balanced squeezed complex. In fact, these complexes have properties similar to squeezed balls as introduced by Kalai and the more general squeezed complexes of \cite{JK:Nagel-Squeezed}. We show that any balanced squeezed complex is vertex-decomposable and that its flag $h$-vector can be read off from the underlying order ideal. Moreover, we describe explicitly its Stanley-Reisner ideal $I_{\Delta_\bal(U)}$. If $U$ is also shifted, we determine the multigraded generic initial ideal of $I_{\Delta_\bal(U)}$ and establish that the balanced squeezed complex $\Delta_\bal(U)$ has the same graded Betti numbers as the complex obtained from color-shifting it. 
      We also introduce a class of color-squarefree monomial ideals that may be viewed as a generalization of the classical squarefree stable monomial ideals and show that their graded Betti numbers can be read off from their minimal generators. Moreover, we develop some tools for computing graded Betti numbers.   
\end{abstract}

\date{July 1, 2020}

\thanks{The first author was partially supported by the German Research Council DFG GRK-1916. The second author was partially supported by Simons Foundation grants \#317096 and \#636513. }

\maketitle

%%%%%%%%%%%%%%%%%%%%%%%%%%%%%%

%\tableofcontents

\section{Introduction}
\label{sect_intro}

In \cite{Kalai}, Kalai introduced a large class of triangulated balls, called \emph{squeezed balls}, with remarkable properties. In \cite{JK:Nagel-Squeezed}, the authors extended Kalai's construction by associating to any shifted order ideal a family of simplical complexes, called \emph{squeezed complexes},  with similar properties. Here we introduce a construction that produces balanced simplicial complexes that still have many of these properties. A balanced complex, as originally introduced by Stanley \cite{Sta79}, is a simplicial complex whose $1$-skeleton admits a ``minimal'' coloring. Examples of balanced complexes include barycentric subdivisions of regular CW-complexes, Coxeter complexes, and Tits buildings. In recent years, they have been studied extensively and many results, in particular from the area of classical face enumeration, have been shown to possess a balanced analogue (see e.g., \cite{GKN,IKN,JKM,JMNS,JV}). 

Given an integer $d >0$ and a vector $\mb=(m_1,\ldots,m_d)$ of non-negative integers, let $P(d,\mb)$ be a polynomial ring with variables $x_{1,1},\ldots x_{1,m_1},\ldots,x_{d,1},\ldots,x_{d,m_d}$ over a field $\K$. We say that a variable $x_{i, j}$ has color $i$. A squarefree monomial $U \in P(d,\mb)$ is said to be \emph{color-squarefree} if  it is divisible by at most one variable of each color. A \emph{color-squarefree monomial order ideal} $U$ of $P(d,\mb)$ is a finite set $U$ of color-squarefree monomials satisfying the following conditions: 
\begin{itemize}
\item[(i)] $U$ is closed under divisibility;%if a monomial divides a monomial in $U$, then this monomial belongs to $U$; 

\item[(ii)]  $1, x_{1,1},\ldots x_{1,m_1},\ldots,x_{d,1},\ldots,x_{d,m_d} \in U$. 
\end{itemize}
In general, Condition (ii) is not assumed. However, it is convenient (and harmless). 

A color-squarefree monomial order ideal $U$ is said to be \emph{shifted} if it also satisfies: 
\begin{itemize}
\item[(iii)] any variable $x_{k, \ell}$ dividing a monomial in $U$ can be replaced by a variable $x_{k, j}$ of the same color with $\ell < j \le m_k$ to obtain another monomial in $U$. 
\end{itemize}

Given \emph{any} color-squarefree monomial order ideal $U$ of $P(d,\mb)$, we construct a balanced $(d-1)$-dimensional simplical complex  $\Delta_{\bal}(U)$  on a vertex set corresponding 
to the variables of $P(d, \mb + (1,\ldots,1))$ and call it the \emph{balanced squeezed complex} to $U$ (see \Cref{Definition:BalancedSqueezed}). It is defined by a description of its facets and shares several 
properties with squeezed balls. More precisely, any balanced squeezed complex $\Delta_{\bal}(U)$  is vertex-decomposable and as such shellable (see \Cref{thm:vertex decomposable}). As a consequence, we show  that $\Delta_{\bal}(U)$ has the same flag $h$-vector as $U$. Namely, for any subset $S$ of $\{1,\ldots,d\}$, the 
number $h_S (\Delta_{\bal}(U))$ counts the number of monomials in $U$ whose color support is $S$. Here, the color support of a monomial  $u=x_{i_1,j_1}\cdots x_{i_s,j_s} \in P(d, \mb)$ with $i_1 < \cdots < i_s$ is defined as $\cSupp (u) = \{i_1,\ldots,i_s\}$.  In addition, we explicitly describe the Stanley-Reisner ideal of $\Delta_{\bal}(U)$ (see \Cref{thm:Stanley-Reisner ideal}). Remarkably, and this is in contrast to the situation for general squeezed balls and complexes, all these results do not require $U$ to be shifted. 

We are also interested in the multigraded generic initial ideal of the Stanley-Reisner ideal of $\Delta_{\bal}(U)$. If $U$ is a shifted color-squarefree monomial  order ideal, the mentioned multigraded gin turns out to be a strongly color-stable ideal in the sense of  Murai \cite{Murai-ColorStable} and can be immediately read off from $U$ (see \Cref{thm:gin}). 

As an extension of the classical algebraic shifting (see, e.g., \cite{Ka-02}), Babson and Novik introduced color-shifting in \cite{Babson:Novik}. It associates to any balanced simplicial complex $\Gamma$ a color-shifted balanced complex $\widetilde{\Gamma}$ by passing from the Stanley-Reisner ideal of $\Gamma$ first to its strongly color-stable multigraded generic initial ideal and then to another squarefree monomial ideal, which is by definition the Stanley-Reisner ideal of $\widetilde{\Gamma}$. By the main result in \cite{Murai-ColorStable}, the $\Z$-graded Betti numbers of $\Gamma$ are bounded above by the $\Z$-graded Betti numbers of $\widetilde{\Gamma}$. We establish that for a color-squarefree shifted monomial order ideal $U$ the $\Z$-graded Betti numbers of its balanced squeezed complex $\Delta_{\bal}(U)$  and the ones of the complex $\widetilde{\Delta_{\bal}(U)}$ obtained by color-shifting coincide (see \Cref{thm:compare with shifted}). 

We also consider generalizations of strongly color-stable ideals and (classical) squarefree stable monomial ideals. More precisely, we introduce color-squarefree monomial ideals that are strongly color-stable across colors or color-squarefree stable across colors. The former class of ideals is defined to satisfy a stronger exchange property than strongly color-stable ideals (see \Cref{sec: prelim} for the precise definition). These ideals are in one-to-one correspondence to color-squarefree monomial order ideals that are shifted across colors, a similar notion that is stronger than just being shifted. We also characterize their balanced squeezed complexes (see \Cref{prop:squeezed across colors}). The second class of ideals may be viewed as a multigraded (or colored) generalization of squarefree stable monomial ideals in the standard setting (of one color). The class of these ideals includes every color-squarefree monomial ideal that is strongly color-stable across colors. As for squarefree stable monomial ideals (see \cite{EK}), we show that the $\mathbb{Z}^d$-graded Betti numbers of any squarefree monomial ideal that is color-squarefree stable across colors can be read off from its minimal generators (see \Cref{cor:Betti numbers color-sqfree stable}). 

Now we describe the organization of this article. After reviewing some basic concepts and results in the following section, we define balanced squeezed complexes in \Cref{sec:squeezed combintorics}. There we establish that any such complex is vertex-decomposable and describe its Stanley-Reisner ideal. These results are used in \Cref{sec:gins} in order to determine explicitly the multigraded generic initial ideal when one starts with a shifted color-squarefree monomial order ideal. In \Cref{sec:Betti} we investigate graded Betti numbers. There we determine the graded Betti numbers of any color-squarefree monomial ideal that is color-squarefree stable  across colors, and we establish the 
preservation of $\Z$-graded Betti numbers when passing from the balanced squeezed complex of a shifted color-squarefree monomial order ideal to its color-shifted complex. The latter result is a 
consequence of a more general fact that may be of independent interest. Namely, given any color-squarefree monomial ideal $I \subset P(d, \mb)$, define a squarefree monomial ideal $I' \subset P(d, \mb + (1,\ldots,1))$ by adding to the extension ideal of $I$ in $P(d, \mb + (1,\ldots,1))$ the ideal generated by all products of two variables of the same color. This ideal $I'$ can be minimally resolved by iterated mapping cones from a suitable long exact sequence (see \Cref{thm:add polarized squares}). A variant of this result (see \Cref{prop:add  squares}) may be viewed as colored version of Theorem 2.1 in \cite{MPS} about resolving the sum of a squarefree monomial ideal and the ideal generated by the squares of all variables. 

%%%%%%%%%%%%%%%%%%%%%%%%%%%%%

\section{Preliminaries}
\label{sec: prelim}

\subsection{Strongly color-stable ideals and color-squarefree shifted order ideals}\label{sect:prel1}
We start by fixing some notation. Given a non-negative integer $d\in \N$ and a vector $\mb=(m_1,\ldots,m_d)\in \N^d$ we write $P(d,\mb)$ for the polynomial ring in variables $x_{1,1},\ldots x_{1,m_1},\ldots,x_{d,1},\ldots,x_{d,m_d}$ over a field $\K$. We often refer to the variables $x_{i,1},\ldots,x_{i,d_i}$ as variables of \emph{color} $i$. We will consider the following total ordering on the variables of $P(d,\mb)$. We let $x_{i,j}\preceq x_{k,\ell}$ if $i<k$ or $i=k$ and $j<\ell$. If one writes the variables into a matrix (adding zeros at the end of rows if necessary), then $x_{k,\ell}$ is larger than all variables to the left or in the same column but below. Furthermore, for a monomial ideal $I$, we denote by $G(I)$ the set of monomial minimal generators of $I$. 

A monomial $u=x_{i_1,j_1}\cdots x_{i_s,j_s}$ is called \emph{color-squarefree} if $i_1<\cdots <i_s$. In other words, $u$ is squarefree and it is divisible by at most one variable of each color. E.g., the monomial $x_{1,1}x_{1,2}$ is squarefree but not color-squarefree. We use $\Mon(d,\mb)$ and $\Mon_{\cs}(d,\mb)$ to denote the set of monomials and  color-squarefree monomials in $P(d,\mb)$, respectively.  Note that whenever $u\in \Mon_{\cs}(d,\mb)$ and $v$ divides $u$, then also $v\in \Mon_{\cs}(d,\mb)$, i.e., the set $\Mon_{\cs}(d,\mb)$ is closed under taking divisors. With this we can define a \emph{color-squarefree monomial order ideal} $U$ to be a monomial order ideal contained in $\Mon_{\cs}(d,\mb)$. In other words, $U$ is a subset of $\Mon_{\cs}(d,\mb)$ that is closed under taking divisors, i.e., whenever $u\in U$ and $v$ divides $u$, then $v\in U$. Observe 
that any non-empty monomial order ideal contains the monomial $1$. Moreover, all color-squarefree monomial order ideals of $\Mon_{\cs}(d,\mb)$ are finite by definition. In the following, we will always assume that a color-squarefree monomial order ideal contains all the variables $x_{1,1},\ldots,x_{1,m_1},\ldots,x_{d,1},\ldots,x_{d,m_d}$. A monomial order ideal 
$U\subset\Mon_{\cs}(d,\mb)$ is called \emph{color-squarefree shifted} if it is color-squarefree and with each monomial $x_{k,\ell}\cdot u\in U$ also the monomial $x_{k,j}\cdot u$ lies in $U$ whenever $j\geq \ell$. 
%and $x_{k,j}\cdot u\in \Mon_{\cs}(d,\mb)$. 
If in addition $x_{i,j}\cdot u$ lies in $U$ for every monomial $x_{k,\ell}\cdot u\in U$ whenever $x_{k,\ell}\preceq x_{i,j}$ and $x_{i,j}\cdot u\in \Mon_{\cs}(d,\mb)$, then $U$ is called \emph{color-squarefree shifted across colors}. In other words, $U$ is shifted with respect to $\preceq$ in the usual sense as long as one does not leave $\Mon_{\cs}(d,\mb)$.  

Following \cite{Murai-ColorStable} we call a monomial ideal $I\subseteq P(d,\mb)$ \emph{strongly color-stable} if for all $x_{k,\ell}\cdot u\in I$ and $j<\ell$ we also have $x_{k,j}\cdot u\in I$. We say that a monomial ideal $I\subseteq P(d,\mb)$ is \emph{strongly color-stable across colors} if it is strongly color-stable and for all $x_{k,\ell}\cdot u\in I$ that are  color-squarefree and $x_{i,j}\preceq x_{k,\ell}$ such that $x_{i,j}\cdot u$ is color-squarefree, we also have $x_{i,j}\cdot u\in I$.  

It is easy to see that color-squarefree monomial order ideals in $P(d,\mb)$ and monomial ideals in $P(d,\mb)$ containing the ideal $(x_{1,1},\ldots,x_{1,m_1})^2+\cdots+ (x_{d,1},\ldots,x_{d,m_d})^2$ are in one-to-one-corres\-pon\-dence. Moreover, similar to the classical situation for shifted monomial order ideals and strongly stable ideals, under this correspondence, color-squarefree shifted monomial order ideals are mapped to strongly color-stable ideals. In order to describe the correspondence  we recall Definition 2.5 from \cite{JK:Nagel-Squeezed}.

\begin{definition}
Let $U\subseteq \Mon(d,\mb)$ be a monomial order ideal. Then $I(U)$ is the monomial ideal of $P(d,\mb)$ that is generated by the monomials in $\Mon(d,\mb)\setminus U$.
\end{definition}

To simplify notation, we set $\mf_i=(x_{i,1},\ldots,x_{i,m_i})$ for $1\leq i\leq d$ and fixed $\mb\in \mathbb{N}^d$. We further denote by $d_{\max}(U)$ and $d_{\max}(I(U))$ the maximal degree of a monomial in $U$ and a minimal generator of $I(U)$, respectively. 
We can now state the mentioned correspondence. 

\begin{lemma} 
     \label{lem:order ideal vs ideal}
Let $U\subseteq P(d,\mb)$ be a monomial order ideal. Then $U$ is color-squarefree shifted (across colors) if and only if $I(U)$ is strongly color-stable (across colors) and $\mf_1^2+\cdots +\mf_d^2\subseteq I(U)$. Moreover,
\begin{equation*}
2\leq d_{\max}(I(U))\leq d_{\max}(U)+1.
\end{equation*}
\end{lemma}

\begin{proof} 
For the first statement, we note that $U$ is color-squarefree if and only if $(\mf_1^2+\cdots + \mf_d^2)\cap U=\emptyset$, i.e., $\mf_1^2+\cdots+ \mf_d^2\subseteq I(U)$. Assume that $U$ is color-squarefree shifted across colors. Let $x_{k,\ell}\cdot u\in I(U)$ be a color-squarefree monomial and let $x_{i,j}\preceq x_{k,\ell}$ be such that $x_{i,j}\cdot u$ is color-squarefree. If $x_{i,j}\cdot u\notin I(U)$, then $x_{i,j}\cdot u\in U$ and since $U$ is color-squarefree shifted across colors, we have $x_{k,\ell}\cdot u\in U$, which yields a contradiction. Hence, $I(U)$ is strongly color-stable across colors. The other direction follows from the same arguments. The statement for $U$ being only color-squarefree shifted follows in the same way.

The second statement is Lemma 2.6 (i) from \cite{JK:Nagel-Squeezed}. 
\end{proof}

\begin{example}
Let $U=\{1,x_{1,1},x_{1,2},x_{2,1},x_{2,2},x_{1,2}x_{2,2}\}\subseteq P(2,(2,2))$. It is easy to check that $U$ is a color-squarefree shifted order ideal with $d_{\max}(U)=2$ and 
\begin{equation}\label{eq:example}
I(U)=\langle x_{1,1}^2,x_{1,2}^2,x_{2,1}^2,x_{2,2}^2,x_{1,1}x_{1,2},x_{1,1}x_{2,1},x_{1,1}x_{2,2},x_{1,2}x_{2,1},x_{2,1}x_{2,2}\rangle.
\end{equation}
We observe that any monomial of degree $3$ monomial in $P(2,(2,2))$ is divisible by one of the degree $2$ monomials on the right-hand side of \eqref{eq:example}. Hence, $d_{\max}(I(U))=2$. 
\end{example}
 The previous example shows that, in contrast to the classical setting, it is not necessarily true that $d_{\max}(I(U))=d_{\max}(U)+1$, if $U$ is color-squarefree shifted \cite[Lemma 2.6 (ii)]{JK:Nagel-Squeezed}. Indeed, as the next example shows, we can even have $d_{\max}(I(U))=2$ independent of $d_{\max}(U)$. 
 
 \begin{example}
The set $\Mon_{\cs}(d,\mb)$ of all color-squarefree monomials of $P(d,\mb)$ is clearly a color-squarefree shifted monomial order ideal with $d_{\max}(\Mon_{\cs}(d,\mb))=d$. Moreover, since any non-color-squarefree monomial of arbitrary degree is divisible by a non-\-co\-lor-square\-free monomial of degree $2$, it follows that the minimal generators of $I(\Mon_{\cs}(d,\mb))$ are the non-color-squarefree monomials in $\Mon(d,\mb)$ of degree $2$, i.e., the degree $2$ monomials in $\Mon(d,\mb)\setminus \Mon_{\cs}(d,\mb)$. Hence, $d_{\max}(I(\Mon_{\cs}(d,\mb)))=2$. 
 \end{example}
 
\subsection{Combinatorics of simplicial complexes}
We recall basic facts on simplicial complexes, including some of their combinatorial properties. We refer to \cite{Stanley-greenBook} for more details.

Given a finite set $V$, an (abstract) \emph{simplicial complex} $\Delta$ on the vertex set $V$ is a collection of subsets of $V$ that is closed under inclusion. Throughout this paper, all simplicial complexes are assumed to be finite. 
Elements of $\Delta$ are called \emph{faces} of $\Delta$ 
and inclusion-maximal faces are called \emph{facets} of $\Delta$. The \emph{dimension} of a face $F \in \Delta$ is its cardinality minus one, and the \emph{dimension} of $\Delta$ is defined as $\dim\Delta:=\max\{\dim F~:~F\in \Delta\}$. $0$-dimensional and $1$-dimensional faces are called \emph{vertices} and \emph{edges}, respectively. 
A simplicial complex $\Delta$ is \emph{pure} if all its facets have the same dimension. The \emph{link} of a face $F\in \Delta$ is the subcomplex
$$
\lk_\Delta(F)=\{G\in\Delta~:~G\cap F=\emptyset,\; G\cup F\in \Delta\}.
$$
We will write $\lk_\Delta(v)$ for the link of a vertex $v$. 
The \emph{deletion} $\Delta\setminus F$ of a face $F\in \Delta$ from $\Delta$ describes the simplicial complex $\Delta$ outside of $F$:
$$
\Delta\setminus F=\{G\in \Delta~:~ F\not\subseteq G\}.
$$
For subsets $F_1,\ldots,F_m$ of a finite set $V$, we denote by $\langle F_1,\ldots,F_m\rangle$ the smallest simplicial complex containing $F_1,\ldots,F_m$, i.e.,
$$
\langle F_1,\ldots,F_m\rangle=\{G\subseteq V~:~G\subseteq F_i\mbox{ for some }1\leq i\leq m\}.
$$
Given a $(d-1)$-dimensional simplicial complex $\Delta$, we let $f_i(\Delta)$ denote the number of $i$-dimensional faces of $\Delta$ and we call $f(\Delta)=(f_{-1}(\Delta), f_0(\Delta), \ldots, f_{d-1}(\Delta))$ the \emph{$f$-vector} of $\Delta$. The \emph{$h$-vector} $h(\Delta)=(h_0(\Delta),h_1(\Delta),\ldots,h_d(\Delta))$ of $\Delta$ is defined via the relation
$$
\sum_{i=0}^{d}f_{i-1}(\Delta)(t-1)^{d-i}=\sum_{i=0}^dh_i(\Delta)t^{d-i}.
$$
We now consider several relevant combinatorial properties of simplicial complexes.

\begin{definition}
Let $\Delta$ be a pure simplicial complex.
\begin{itemize}
\item[(i)] $\Delta$ is called \emph{shellable} if there exists an ordering $F_1,\ldots,F_n$ of the facets of $\Delta$ such that for each $1< i\leq n$ the intersection $\langle F_i\rangle\cap\left(\bigcup_{j=1}^{i-1}\langle F_j\rangle\right)$ is generated by a non-empty set of maximal proper faces of $\langle F_i\rangle$. 
\item[(ii)] $\Delta$ is called \emph{vertex-decomposable} if either $\Delta=\{\emptyset\}$ or there exists a vertex $v\in \Delta$, a so-called \emph{shedding vertex}, such that $\lk_\Delta(v)$ and $\Delta\setminus \{v\}$ are vertex-decomposable. 
\end{itemize}
\end{definition}

It is well-known \cite{Provan:Billera} that vertex-decomposability is a stronger property than shellability.

%If $\Delta$ is a $(d-1)$-dimensional vertex decomposable simplicial complex with shedding vertex $v$, then the face numbers of $\Delta$ can be easily computed via the following recursion:
%$$
%f_i(\Delta)=f_i(\Delta\setminus v)+f_{i-1}(\lk(v)) \quad \mbox{for}\quad 0\leq i\leq d-1.
%$$
%Using the relations between $f$-vectors and $h$-vectors this immediately translates into:
%\begin{equation}\label{h-vector:vertexdecomposable}
%h_i(\Delta)=h_i(\Delta\setminus v)+h_{i-1}(\lk(v)) \quad \mbox{for}\quad 0\leq i\leq d.
%\end{equation}

The main focus in this article lies on balanced complexes, which were originally introduced by Stanley under the name \emph{completely balanced} complexes \cite{Sta79}. We say that a $(d-1)$-dimensional simplicial complex $\Delta$ on vertex set $V$ is \emph{balanced} if its $1$-skeleton is $d$-colorable, i.e., there is a ($d$-coloring) map $\kappa: V\to [d]$ such that $\kappa(u)\neq \kappa(v)$ for any edge $\{u,v\}\in \Delta$. In other words, an $r$-coloring corresponds to a partition $V_1\dot\cup\cdots \dot\cup V_r$ of $V$ such that $\#(F\cap V_i)\leq 1$ for all $F\in \Delta$ and $1\leq i\leq r$.  We often refer to the set $V_i$ as set of vertices of color $i$. From now on, we will always assume that $\Delta$ is endowed with a fixed coloring map $\kappa:V\to [d]$ and use this map. For a subset $S\subseteq [d]$, we set
$$
f_S(\Delta)=\#\{F\in \Delta~:~\kappa(F)=S\},
$$
where $f_\emptyset(\Delta)=1$, and 
$$
h_S(\Delta)=\sum_{T\subseteq S}(-1)^{\#(S\setminus T)}f_T(\Delta).
$$
The vectors $(f_S(\Delta)~:~S\subseteq [d])$ and $(h_S(\Delta)~:~S\subseteq [d])$ are called the \emph{flag $f$-vector} and the \emph{flag $h$-vector} of $\Delta$, respectively. The usual $f$- and $h$-vector can be recovered from their flag counterparts as follows:
$$
f_{i-1}(\Delta)=\sum_{S\subseteq [d],\#S=i}f_S(\Delta) \quad \text{and} \quad h_i(\Delta)=\sum_{S\subseteq [d],\#S=i}h_S(\Delta)
$$ for $0\leq i\leq d$. 
In the following, we consider a $(d-1)$-dimensional balanced, vertex-decomposable simplicial complex $\Delta$ with shedding vertex $v$. Restricting the $d$-coloring of $\Delta$ to the vertex set of $\lk_{\Delta}(v)$ and $\Delta \setminus v$, it is obvious that $\lk_{\Delta}(v)$ is balanced and so is $\Delta\setminus v$, either of dimension $d-1$ or $d-2$. It is easy to see that the flag $f$-vector of $\Delta$ can be computed via the following recursion: For $\emptyset\neq S\subseteq [d]$ one has
\begin{equation}\label{f-vector:vertexdecomposable}
f_S(\Delta)=\begin{cases}
f_S(\Delta\setminus \{v\})+f_S(\lk_\Delta(v)), \quad &\mbox{if }\kappa(v)\notin S\}\\
f_S(\Delta\setminus \{v\})+f_{S\setminus\{\kappa(v)\}}(\lk_\Delta(v)), \quad &\mbox{otherwise},
\end{cases}
\end{equation}
which simplifies to 
$$
f_S(\Delta)=f_S(\Delta\setminus \{v\})+f_{S\setminus\{\kappa(v)\}}(\lk_\Delta(v))\quad  \text{for any } \emptyset\neq S\subseteq [d].
$$
Using the relations between flag $f$- and $h$-vectors this immediately translates into:
\begin{equation}\label{h-vector:vertexdecomposable}
h_S(\Delta)=h_S(\Delta\setminus \{v\})+h_{S\setminus\{\kappa(v)\}}(\lk_\Delta(v)) \quad  \text{for any } \emptyset\neq S\subseteq [d].
\end{equation}
We want to remark that \eqref{f-vector:vertexdecomposable} and \eqref{h-vector:vertexdecomposable} have straightforward specializations to the usual $f$- and $h$-vectors.

As in the previous section, let $d\in \N$ and $\mb\in \N^d$. For $1\leq i \leq d$ let $V_i=\{1^{(i)},2^{(i)},\ldots ,(m_i+1)^{(i)}\}$ and let $V=\bigcup_{i=1}^d V_i$. Following \cite{Babson:Novik} and \cite{Murai-ColorStable}, we call a balanced simplicial complex $\Delta$ on vertex set $V=\bigcup_{i=1}^d V_i$ \emph{color-shifted} if whenever $F\in \Delta$ and $j^{(i)}\in F$ it also holds that $F\setminus\{j^{(i)}\}\cup \{k^{(i)}\}\in \Delta$ for all $m_{i}+1\geq k\geq j$. If, in addition $F\setminus \{j^{(i)}\}\cup\{k^{(\ell)}\}\in \Delta$ for all $d\geq \ell > i$ and all $1\leq k\leq m_\ell +1$, then $\Delta$ is called \emph{color-shifted across colors}. 

 Given a $(d-1)$-dimensional balanced simplicial complex  $\Delta$ on vertex set $V$ (as above) with coloring map $\kappa:V\to [d]$ there is a natural way to associate a finite set $U(\Delta)$ of color-squarefree monomials to it by taking the monomials given by 
 ${\displaystyle \prod_{j^{(i)} \in F}x_{i,j}}$ for $F\in \Delta$. Obviously, this construction can be reversed and $U(\Delta)$ is a monomial order ideal with
 \begin{equation}
\label{eq:boringCorrespondence}
f_i(\Delta)=\#\{u\in U(\Delta)~:~\deg(u)=i+1\} \qquad \mbox{for all} -1\leq i\leq d-1.
\end{equation}
The aim of this article is to provide and study a construction that,  given a  color-squarefree order ideal  $U$ on $d$ color classes, generates a $(d-1)$-dimensional balanced simplicial complex $\Delta_{\bal}(U)$ such that the right-hand side in \eqref{eq:boringCorrespondence} equals $h_{i+1}(\Delta_{\bal}(U))$. This will be the content of the next section.

\section{Balanced squeezed complexes and their combinatorial properties}
     \label{sec:squeezed combintorics}

In this section, given a shifted color-squarefree monomial order ideal $U\subset\Mon_{\cs}(d,\mb)$, we construct from it a $(d-1)$-dimensional balanced simplicial complex $\Delta_{\bal}(U)$ on vertex set $V(\Delta_{\bal})=\bigcup_{i=1}^d V_i$, where $V_i$ is the set of vertices of color $i$ and $\#V_i=m_i+1$ for $1\leq i\leq d$. We call $\Delta_{\bal}(U)$ the \emph{balanced squeezed complex} associated to $U$. The name is inspired by the so-called squeezed balls due to Kalai \cite{Kalai} and more generally, the squeezed complexes from \cite{JK:Nagel-Squeezed}, which can be associated to any shifted order ideal and which will be seen to share a lot of common properties with our new construction. One of the main differences however is that in the balanced setting, most of the results hold without the assumption of shiftedness. We will study combinatorial properties of those complexes in \Cref{Section:Combinatorics} and their Stanley-Reisner ideals in \Cref{Section:SR-ideals}. 

\subsection{Combinatorics of balanced squeezed complexes}\label{Section:Combinatorics}
We start with the crucial definition.
\begin{definition}\label{Definition:BalancedSqueezed}
Let $d\in \N$ and $\mb=(m_1,\ldots,m_d)\in \N_{\geq 0}^d$. Let $V_i=\{1^{(i)},2^{(i)},\ldots,(m_i+1)^{(i)}\}$ for $1\leq i\leq d$. 
\begin{itemize}
\item[(i)] For $u=x_{i_1,j_1}\cdots x_{i_s,j_s}\in \Mon_{\cs}(d,\mb)$ let $$
F_d(u)=\{j_1^{(i_1)},\ldots,j_s^{(i_s)}\}\cup\{(m_j+1)^{(j)}~:~j\in [d]\setminus\{i_1,\ldots,i_s\}\}\subseteq \bigcup_{i=1}^d V_i.$$
\item[(ii)]  Let $U\subseteq \Mon_{\cs}(d,\mb)$ be a finite set of color-squarefree monomials. Define $\Delta_{\bal}(U)$ to be the $(d-1)$-dimensional simplicial complex,  whose facets are given by $F_d(u)$ for $u\in U$. We call $\Delta_{\bal}(U)$ the  \emph{balanced complex} associated to $U$. 
\item[(iii)] If $U$ is a  color-squarefree order ideal, then we say that $\Delta_{\bal}(U)$ is \emph{balanced squeezed}.
\end{itemize}
\end{definition}

We illustrate the previous definition with an example which will serve as our running example throughout this article. 

\begin{example}\label{RunningExample}
Let $U$ be the smallest shifted color-squarefree monomial order ideal in $P(3,(2,2,2))$ containing $x_{1,2}x_{2,2}x_{3,2}$, i.e.,
\begin{equation*}
U=\{1,x_{1,1},x_{1,2},x_{2,1},x_{2,2},x_{3,1},x_{3,2},x_{1,2}x_{2,2},x_{1,2}x_{3,2},x_{2,2}x_{3,2},x_{1,2}x_{2,2}x_{3,2}\}.
\end{equation*}
Writing $abc$ for $\{a,b,c\}$, the facets of $\Delta_{\bal}(U)$ are given by 
\begin{align*}
&3^{(1)}3^{(2)}3^{(3)}, \quad 1^{(1)}3^{(2)}3^{(3)},\quad 2^{(1)}3^{(2)}3^{(3)},\quad
 3^{(1)}1^{(2)}3^{(3)},\quad 3^{(1)}2^{(2)}3^{(3)}, \quad 3^{(1)}3^{(2)}1^{(3)},\\
 &3^{(1)}3^{(2)}2^{(3)}, \quad 2^{(1)}2^{(2)}3^{(3)}, \quad 2^{(1)}3^{(2)}2^{(3)},\quad
 3^{(1)}2^{(2)}2^{(3)}, \quad 2^{(1)}2^{(2)}2^{(3)},
\end{align*}
where we list the facets in the order corresponding to the order of the monomials in $U$.
\end{example}

The next easy lemma justifies the occurrence of the word \emph{balanced} in part (ii) and (iii) of \Cref{Definition:BalancedSqueezed}.

\begin{lemma}\label{lemma:balanced}
Let $d\in \N$, $\mb=(m_1,\ldots,m_d)\in \N^d$ and let $U\subseteq \Mon_{\cs}(d,\mb)$ be a finite set of color-squarefree monomials. Then, 
$\Delta_{\bal}(U)$ is a $(d-1)$-dimensional balanced simplicial complex. The vertices of $\Delta_\bal(U)$ of color $i$ form a (possibly proper) subset of $V_i$.
\end{lemma}

\begin{proof}
As $\#F_d(u)=\deg(u)-(d-\deg(u))=d$ we have $\dim\Delta_{\bal}(U)=d-1$. Moreover, by definition of $F_d(u)$ we have that $\#(F_d(u)\cap V_i)=1$ for any $u\in U$ and $1 \leq i\leq d$, which directly implies that $\Delta_{\bal}(U)$ is balanced.
\end{proof}

\begin{example}\label{example:crosspolytope}
Let $\mb = (1,\ldots,1) \in \N^d$ be the vector whose entries are all equal to one,  and let $U\subseteq P(d,\mb)$ be the monomial order ideal containing all color-squarefree monomials in $\Mon(d,\mb)$, i.e., $U=\Mon_{\cs}(d,\mb)$. We note that $U$ is color-squarefree shifted. It is easy to see that the balanced squeezed complex associated to $U$ is the boundary complex of the $d$-dimensional cross-polytope on vertex set $\{1^{(1)},2^{(1)},\ldots, 1^{(d)},2^{(d)}\}$. Indeed, the only non-faces of the $d$-dimensional cross-polytope $C_d$ on vertex set $\{1^{(1)},2^{(1)},\ldots, 1^{(d)},2^{(d)}\}$ are those, which would destroy balancedness, i.e., $\{1^{(i)},2^{(i)}\}$ for $1\leq i\leq d$. Since $\Delta_{\bal}(U)$ is balanced by \Cref{lemma:balanced} those are also missing faces of $\Delta_\bal(U)$, which implies that $\Delta_\bal(U)$ is a subcomplex of $C_d$. Moreover, we have
$$
F_d(x_{i_1,1}\cdots x_{i_s,1})=\{1^{(i_1)},\ldots,1^{(i_s)}\}\cup\{2^{(j)}~:~j\in [d]\setminus \{i_1,\ldots,i_s\}\},
$$
from which it follows that any face of $C_d$ belongs to $\Delta_{\bal}(U)$.

This implies that any color-squarefree order ideal in $P(d,\mb)$ gives rise to a balanced complex, that is a subcomplex of $C_d$ on vertex set $\{1^{(1)},2^{(1)},\ldots, 1^{(d)},2^{(d)}\}$.
\end{example}
Before we focus on the case of color-squarefree order ideals, we collect several immediate properties of $\Delta_{\bal}(U)$ for general $U$.

\begin{lemma}\label{Lemma:easyProperties}
Let $d\in \N$, $\mb=(m_1,\ldots,m_d)\in \N_{\geq 0}^d$ and let $\emptyset\neq U\subseteq \Mon_{\cs}(d,\mb)$ be a finite set of color-squarefree monomials. Then
\begin{itemize}
\item[(i)] $f_{d-1}(\Delta_{\bal}(U))=h_0(\Delta_{\bal}(U))+\cdots +h_d(\Delta_{\bal}(U))=\#U$.
\item[(ii)] Every element of $\bigcup_{i=1}^d V_i$ is a vertex of  $\Delta_{\bal}(U)$  if and only if for each variable $x_{i,j} \in P(d, \mb)$ there exists a monomial $u\in U$ that is divisible by $x_{i,j}$ and for each color $1\leq i\leq d$ there exists a monomial $u\in U$ not divisible by any variable of color $i$.
\end{itemize}
\end{lemma}

\begin{proof}
Part (i) is immediate since $F_d(u)\neq F_d(v)$ for $u,v\in U$ with $u\neq v$. 

For part (ii),  let $j^{(i)}\in V_i$. We have $j^{(i)}\in \Delta_{\bal}(U)$ if and only if there exists $u=x_{i_1,j_1}\cdots x_{i_s,j_s}\in U$ with
$$
j^{(i)}\in F_d(u)=\{j_1^{(i_1)},\ldots,j_s^{(i_s)}\}\cup\{(m_\ell+1)^{(\ell)}~:~\ell\in [d]\setminus\{i_1,\ldots,i_s\}\}.
$$
If $1\leq j\leq m_{i}$ this is the case if and only if $j^{(i)}\in \{j_1^{i_1},\ldots,j_s^{(i_s)}\}$, i.e., $x_{i,j}$ divides $u$. If $j=m_i+1$, then $j^{(i)}\in \{(m_\ell+1)^{(\ell)}~:~\ell\in [d]\setminus\{i_1,\ldots,i_s\}\}$, i.e., $u$ is not divisible by any variable of color $i$. 
\end{proof} 

We note that the balanced simplicial complexes that are obtained from  \Cref{Definition:BalancedSqueezed} when $U$ satisfies the conditions in  \Cref{Lemma:easyProperties} (ii) will have a cone point if and only if there exists a color class in $P(d,\mb)$ of size $0$, that is there exists $1\leq i\leq d$ with $m_i=0$. Though from a combinatorial point of view this case is hence rather uninteresting, for technical reasons it will be convenient to allow $m_i=0$ (e.g., to simplify the statement in \Cref{lemma:deletion}). In the following, we will restrict our attention to the situation in \Cref{Lemma:easyProperties} (ii). This includes any finite color-squarefree set $U\subseteq\Mon_{\cs}(d,\mb)$ with $1\in U$ and $x_{i,j}\in U$ for all $1\leq i\leq d $ and $1\leq j\leq m_i$ and in particular all finite color-squarefree order ideals (containing all the variables). 

%In the following, we will be particularly interested in the case that $U$ is a shifted color-squarefree order ideal.

It was shown in \cite{Kalai} and \cite{JK:Nagel-Squeezed} that squeezed and more generally $t$-squeezed complexes are vertex decomposable and that their $h$-vector is given by the right-hand side of \eqref{eq:boringCorrespondence}. In the following, we will show that this is also true for balanced squeezed complexes, even if $U$ is an order ideal which is not necessarily shifted. Moreover, the corresponding \emph{multigraded} version of this statement on the level of flag $h$-vectors is also true. The proof strategy follows the same lines as in \cite{JK:Nagel-Squeezed}.

\begin{lemma}\label{lemma:link}
Let $d\geq 2$, $U\subseteq P(d,\mb)$ be a color-squarefree order ideal and let $\widetilde{U}=\{u\in U~:~x_{1,1}\cdot u\in U\}$. Then 
$$
\lk_{\Delta_{\bal}(U)}(1^{(1)})=\langle F_d(u)\setminus \{(m_1+1)^{(1)}\}~:~u\in \widetilde{U}\rangle.
$$
In particular, $\lk_{\Delta_{\bal}(U)}(1^{(1)})$ is combinatorially isomorphic to the balanced complex of $\widetilde{U}$ considered as color-squarefree order ideal in $P(d-1,(m_2,\ldots,m_{d}))$.% with respect to the induced ordering of $\preceq$. 
\end{lemma}

\begin{proof}
We first note that since $\Delta_{\bal}(U)$ is pure, so is $\lk_{\Delta_{\bal}(U)}(1^{(1)})$. A facet $F_d(u)\in \Delta_{\bal}(U)$ contains the vertex $1^{(1)}$ if and only if $x_{1,1}$ divides $u$. Hence,
$$
\lk_{\Delta_{\bal}(U)}(1^{(1)})=\langle F_d(u\cdot x_{1,1})\setminus \{1^{(1)}\}~:~u\in \widetilde{U}\rangle.
$$
Since $F_d(u\cdot x_{1,1})\setminus \{1^{(1)}\}=F_d(u)\setminus \{(m_1+1)^{(1)}\}$, the claim follows. As every monomial in $U$ is divisible by at most one variable of color $1$, the ``In particular''-part holds. 
\end{proof}

We remark that the proof of \Cref{lemma:link} works for any vertex $i^{(\ell)}$ with $1\leq \ell \leq d$ and $1\leq i\leq m_\ell$. In other words,
$$
\lk_{\Delta_{\bal}(U)}(i^{(\ell)})=\langle F_d(u)\setminus \{(m_\ell+1)^{(\ell)}\}~:~u\in U \mbox{ such that }x_{\ell,i}\cdot u\in U\rangle.
$$

For the deletion the following statement is true.
\begin{lemma}\label{lemma:deletion}
Let $U\subseteq P(d,\mb)$ be a color-squarefree order ideal. Let $i$ such that $m_i\geq 1$ and let $\hat{U}=\{u\in U~:~x_{i,1}\nmid u\}$. Then 
$$
\Delta_{\bal}(U)\setminus \{1^{(i)}\}=\langle F_d(u)~:~u\in \hat{U}\rangle.
$$
In particular, $\Delta_{\bal}(U)\setminus \{1^{(1)}\}$ is isomorphic to the balanced complex of $\hat{U}$ considered as color-squarefree order ideal in $P(d,(m_1,\ldots,m_{i-1},m_i-1,m_{i+1},\ldots,m_d))$.
\end{lemma}

\begin{proof}
Let $u\in \hat{U}$, i.e., $x_{i,1}\nmid u$. Since $m_i\geq 1$, it follows that $1^{(i)}\notin F_d(u)$ and hence $F_d(u)\in \Delta_{\bal}(U)\setminus \{1^{(i)}\}$. This shows the containment of the right-hand side in the left-hand side of the above equation.

In order to show the reverse containment it suffices to show that $\Delta_{\bal}(U)\setminus \{1^{(i)}\}$ is pure and of the same dimension as $\Delta_{\bal}(U)$. First note that, $1\in \hat{U}$ and by the just shown containment $F_d(1)\in \Delta_{\bal}(U)\setminus \{1^{(i)}\}$. Hence, $\dim(\Delta_{\bal}(U)\setminus \{1^{(i)}\})=\dim\Delta_{\bal}(U)=d-1$, as required. Assume by contradiction that there exists a facet $F\in \Delta_{\bal}(U)\setminus \{1^{(i)}\}$ with $\dim F<d-1$. As $\Delta_{\bal}(U)\setminus\{ 1^{(i)}\}$ is a subcomplex of $\Delta_{\bal}(U)$, there exists $u\in U\setminus \hat{U}$ such that $F\subsetneq F_d(u)$. In particular, $x_{i,1}$ divides $u$, i.e., $u=x_{i,1}\cdot \tilde{u}$ for some monomial $\tilde{u}$. As $U$ is color-squarefree and an order ideal, it further holds that $x_{i,1}\nmid\tilde{u}$ and $\tilde{u}\in U$, respectively. We thus have $\tilde{u}\in \hat{U}$. Moreover, we have
 \begin{equation}\label{eq:FacetDiff}
F_d(\tilde{u})=\left(F_d(u)\setminus \{1^{(i)}\}\right)\cup \{(m_i+1)^{(i)}\}
\end{equation}
and $F_d(\tilde{u})\in \Delta_\bal(U)$ by the containment already shown. 
Since $1^{(i)}\notin F$, $1^{(i)}\in F_d(u)$ and $F\subsetneq F_d(u)$, we deduce that $F\subseteq F_d(u)\setminus \{1^{(i)}\}$ and by \eqref{eq:FacetDiff} it follows that $F\subseteq F_d(\tilde{u})$, which contradicts the maximality of $F$.
\end{proof}

The next statement is an almost immediate consequence of the previous two lemmas.

\begin{theorem}\label{thm:vertex decomposable}
Let $U\subseteq P(d,\mb)$ be a color-squarefree order ideal. Then $\Delta_{\bal}(U)$ is vertex-decomposable and in particular shellable. In particular, if $m_i=0$ for all $1\leq i\leq d$, then $\Delta_{\bal}(U)$ is a $(d-1)$-simplex. Otherwise, any vertex $1^{(i)}$ with $m_i\geq 1$ can be taken as a shedding vertex.
\end{theorem}

\begin{proof}
We use double induction on $m_1+\cdots +m_d$ and $d$. If $d=1$, then $U=\{1,x_{1,1},\ldots,x_{1,m_1}\}$ and $\Delta_{\bal}(U)$ consists of $m_1+1$ isolated vertices. If $m_1+\cdots+m_d=0$, then $U=\{1\}$ and $\Delta_{\bal}(U)$ is just a $(d-1)$-simplex.

Let $m_1+\cdots +m_d\geq 1$ and $d\geq 2$. Then there exists $1\leq i\leq d$ such that $m_i\geq 1$. By \Cref{lemma:link} $\lk_{\Delta_{\bal}(U)}(1^{(i)})$ is isomorphic to the balanced complex associated to the order ideal $\widetilde{U}=\{u\in U~:~x_{i,1}\cdot u\in U\}$ in $P(d-1,(m_1,\ldots,m_{i-1},m_{i+1},\ldots,m_d))$. The induction hypothesis implies that $\lk_{\Delta_{\bal}(U)}(1^{(i)})$ is vertex-decomposable. 
By \Cref{lemma:deletion} $\Delta_{\bal}(U)\setminus\{ 1^{(i)}\}$ is isomorphic to the balanced complex associated to the order ideal $\hat{U}=\{u\in U~:~x_{i,1}\nmid u\}$ in $P(d,(m_1,\ldots,m_{i-1},m_i-1,m_{i+1},\ldots,m_d))$. Again, the latter complex is vertex-decomposable by the induction hypothesis. The claim follows.
\end{proof}

\begin{example}
Let $\mb\in\N^d$ be the all ones vector. It follows from \Cref{example:crosspolytope} and \Cref{thm:vertex decomposable} that for any color-squarefree order ideal $U$ in $P(d,\mb)$, the balanced complex $\Delta_{\bal}(U)$ is a full-dimensonal subcomplex of $C_d$ which is vertex-decomposable. 
\end{example}

For $u\in \Mon(d,\mb)$ we set 
\begin{equation}\label{eq:multidegree}
\cSupp (u) =\{i~:~\mbox{ there exists }j \mbox{ such that }x_{i,j}\mbox{ divides } u\}
\end{equation}
and call this the  \emph{color support} of $u$. 
\Cref{thm:vertex decomposable} allows us to derive the following expression for the flag $h$-vector of $\Delta_{\bal}(U)$.

\begin{corollary}\label{cor:h-vector}
Let $U\subseteq P(d,\mb)$ be a color-squarefree order ideal. Then
\begin{equation}\label{eq:h-vector}
h_S(\Delta_{\bal}(U))=\#\{u\in U~:~\cSupp (u)=S\} \quad \mbox{for any } S\subseteq [d].
\end{equation}
\end{corollary}

\begin{proof}
The proof proceeds by induction on $m_1+\cdots+m_d$. For $m_1+\cdots+m_d=0$, we have seen in \Cref{{thm:vertex decomposable}} that $U=\{1\}$ and $\Delta_{\bal}(U)$ is a $(d-1)$-simplex. In particular, $h_S(\Delta_{\bal}(U))=1$ if $S=\emptyset$ and $h_S(\Delta_{\bal}(U))=0$, otherwise. This shows the base of the induction. We omit the induction step since it is almost verbatim the same as the proof of Corollary 3.9 in \cite{JK:Nagel-Squeezed}, using the flag $h$-vector together with \eqref{h-vector:vertexdecomposable} instead of the usual $h$-vector.
\end{proof}

\begin{example}
We consider the balanced squeezed complex $\Delta_\bal(U)$ from \Cref{RunningExample}. Let $S\subseteq [3]$. On the one hand, we compute
\begin{align*}
h_S(\Delta_\bal(U))=\begin{cases}
1, \quad &\mbox{if } \#S\in\{0,2,3\}\\
=2, \quad &\mbox{if }\#S=1.
\end{cases}
\end{align*}
On the other hand, we note that $U$ contains $2$ monomials with color support $S$ if $\#S=1$ and just one such monomial otherwise.
\end{example}

\begin{remark}\label{rem:h-vector}
We note that since $h_i(\Delta_{\bal}(U))=\sum_{S\subseteq[d],\#S=i}h_i(\Delta_{\bal}(U))$, Equation \eqref{eq:h-vector} implies
$$
h_i(\Delta_{\bal}(U))=\#\{u\in U~:~\deg(u)=i\} \quad \mbox{for } 0\leq i\leq d,
$$
where $\deg(u)$ denotes the usual degree of a monomial. 
\end{remark}

We end this section with a last observation.

\begin{lemma}\label{lem:compareIdealAndComplex}
Let $U\subseteq P(d,\mb)$ be a color-squarefree monomial order ideal. Then, $U$ is color-squarefree shifted if and only if the balanced squeezed complex $\Delta_{\bal}(U)$ is color-shifted. 
\end{lemma}

\begin{proof}
We first assume that $U$ is color-squarefree shifted. 
Let $F\in \Delta_{\bal}(U)$. As $\Delta_{\bal}(U)$ is pure, there exists a facet $G$ of dimension $d-1$ in $\Delta$ with $F\subseteq G$. Then there exists $u=x_{i_1,j_1}\cdots x_{i_s,j_s}\in U$ with $1\leq i_1<\cdots <i_s\leq d$ and $1\leq j_\ell\leq m_{i_\ell}$ such that 
$$
G=F_d(u)=\{j_1^{(i_1)},\ldots,j_s^{(i_s)}\}\cup\{(m_\ell+1)^{(\ell)}~:~\ell\in [d]\setminus \{i_1,\ldots,i_s\}\}.
$$
 For a fixed $\ell$ with $1\leq \ell\leq s$ with $j_\ell^{(i_\ell)}\in F$, let $m_{i_\ell}+1\geq r\geq j_\ell$. As $\Delta_\bal(U)$ is balanced, we have $r^{(i_\ell)}\notin F$. First assume that $r\leq m_{i_\ell}$. Since $U$ is color-squarefree shifted, we have $v=x_{i_\ell,r}\cdot\frac{u}{x_{i_\ell,j_\ell}}\in U$ . Hence,
 $$
F_d(v)=\left(G\setminus  \{j_\ell^{(i_\ell)}\}\right)\cup\{r^{(i_\ell)}\}\in \Delta_{\bal}(U)
 $$
and in particular, $\left(F\setminus  \{j_\ell^{(i_\ell)}\}\right)\cup\{r^{(i_\ell)}\}\in \Delta_{\bal}(U)$. 
Now assume that $r=m_{i_\ell}+1$. Then $w=\frac{u}{x_{i_\ell,j_\ell}}\in U$, since $U$ is an order ideal.
We conclude that
$$
F_d(w)=\left(G\setminus  \{j_\ell^{(i_\ell)}\}\right)\cup\{(m_{i_\ell}+1)^{i_\ell}\}\in \Delta_{\bal}(U)
$$
and hence $\left(F\setminus  \{j_\ell^{(i_\ell)}\}\right)\cup\{m_{i_\ell}^{(i_\ell)}+1\}\in \Delta_{\bal}(U)$. This shows that $\Delta_\bal(U)$ is color-shifted. The other direction follows from similar arguments.
\end{proof}

\begin{example}
Let $U=\{1,x_{1,1},x_{2,1},x_{3,1}\}$. Then $U$ is color-squarefree shifted and even color-squarefree shifted across colors. $\Delta_\bal(U)$ is the simplicial complex on vertex set $\{1^{(1)},2^{(1)},1^{(2)},$\\
$\{2^{(2)},1^{(3)},2^{(3)}\}$ with facets $\{2^{(1)},2^{(2)},2^{(3)}\}$, $\{1^{(1)},2^{(2)},2^{(3)}\}$, $\{2^{(1)},1^{(2)},2^{(3)}\}$ and $\{2^{(1)},2^{(2)},1^{(3)}\}$.
It is easy to see that $\Delta_\bal(U)$ is color-shifted. However, it is not color-shifted across colors. For instance, the edge $\{1^{(1)},2^{(2)}\}$ would force the edge $\{1^{(1)},1^{(3)}\}$ to be present.
\end{example}

\subsection{Stanley-Reisner ideals of squeezed complexes}\label{Section:SR-ideals}
The aim of this section is to study the Stanley-Reisner ideals of the balanced squeezed complexes. In particular, we provide an explicit description of their minimal generating set. This result is in analogy with the results in \cite{JK:Nagel-Squeezed,Murai-Squeezed}. We will focus on the case of color-squarefree order ideals. 

The description of the minimal generating set $G(I_{\Delta_{\bal}(U)})$ turns out to be easier than in the case of general squeezed complexes. Moreover, it does not require the order ideal $U$ to be shifted.

\begin{theorem} \label{thm:Stanley-Reisner ideal}
Let $d\geq 2$, $\mb\in \N_{\geq 1}^d$ and let $U\subseteq P(d,\mb)$ be a color-squarefree order ideal. Then
\begin{equation}\label{eq:Stanley-Reisner ideal}
I_{\Delta_{\bal}(U)}=\langle u~:~u\in I(U) \mbox{ squarefree }\rangle+\langle x_{i,\ell}\cdot x_{i,m_i+1}~:~1\leq i\leq d,\;1\leq \ell\leq m_i\rangle.
\end{equation}
\end{theorem}
Morally, the previous theorem tells us that $I_{\Delta_{\bal}(U)}$ coincides with the squarefree part of $I(U)$ considered as ideal in $P(d,\mb+(1,\ldots,1))$ plus some obvious minimal generators, which belong to $I_{\Delta_{\bal}(U)}$ due to balancedness and which involve one of the variables $x_{i,m_{i}+1}$ for $1\leq i\leq d$.

%Before we can provide the proof of \Cref{thm:Stanley-Reisner ideal} we need some preparations. 

\begin{proof}
To simplify notation, we set 
$$
J=\langle u~:~u\in I(U) \mbox{ squarefree }\rangle+\langle x_{i,\ell}\cdot x_{i,m_i+1}~:~1\leq i\leq d,\;1\leq \ell\leq m_i\rangle.
$$  
We first show that $J\subseteq I_{\Delta_{\bal}(U)}$. Let $u=x_{i_1,j_1}\cdots x_{i_s,j_s}\in I(U)$ be a squarefree minimal generator with $1\leq i_1<\cdots <i_s\leq d$. If $u\notin I_{\Delta_{\bal}(U)}$, then we would have $\{j_1^{(i_1)},\ldots,j_s^{(i_s)}\}\in \Delta_{\bal}(U)$ and there would exist $\tilde{u}\in U$ such that $\{j_1^{(i_1)},\ldots,j_s^{(i_s)}\}\subseteq F_d(\tilde{u})$. As $u$ does not involve the variables $x_{1,m_1+1},\ldots,x_{d,m_d+1}$, this implies that $u$ divides $\tilde{u}$. Since $U$ is an order ideal, it follows that $u\in U$, which is a contradiction. Moreover, as by \Cref{lemma:balanced} $\Delta_{\bal}(U)$ is balanced with $\{1^{(i)},\ldots,(m_i+1)^{(i)}\}$ being the vertices of color $i$, it follows directly that $\{\ell^{(i)},(m_i+1)^{(i)}\}\notin \Delta_{\bal}(U)$, i.e., $x_{i,\ell}\cdot x_{i,m_{i}+1}\in I_{\Delta_{\bal}(U)}$.

We show the reverse containment in \eqref{eq:Stanley-Reisner ideal} by induction on $\#U$. If $\#U=(m_1+1)\cdots (m_d+1)$, i.e., $U$ contains all color-squarefree monomials in $P(d,\mb)$, then it follows, as in \Cref{example:crosspolytope}, that the balanced complex $\Delta_{\bal}(U)$ is the clique complex of the complete $d$-partite graph on vertex set $\{1^{(1)},\ldots,(m_1+1)^{(1)}\}\cup\cdots\cup\{1^{(d)},\ldots,(m_d+1)^{(d)}\}$. Thus, every minimal  non-face is of the form $\{i^{(j)},\ell^{(j)}\}$, where $1\leq j\leq d$ and $1\leq i<\ell\leq m_j+1$. Moreover, since, by assumption on $U$, $I(U)$ only contains non-color-squarefree monomials, the claim follows. 
Now, let $\#U<(m_1+1)\cdots (m_d+1)$. Then there exists $v\in \Mon_{\cs}(d,\mb)\setminus U$. Moreover, we can choose $v\in \Mon_{\cs}(d,\mb)\setminus U$ such that $\widetilde{U}=U\cup\{v\}$ is a color-squarefree order ideal. Let $v=x_{i_1,j_1}\cdots x_{i_s,j_s}$ with $1\leq i_1<\cdots <i_s\leq d$ and $1\leq j_k\leq m_{i_k}$ for $1\leq k\leq s$. To simplify notation, we set
$J_{\widetilde{U}}=\langle u~:~u\in I(\widetilde{U}) \mbox{ squarefree }\rangle+\langle x_{i,\ell}\cdot x_{i,m_i+1}~:~1\leq i\leq d,\;1\leq \ell\leq m_i\rangle$. By induction hypothesis, we have that $I_{\Delta_{\bal}(\widetilde{U})}\subseteq J_{\widetilde{U}}$. Since $G(J_{\widetilde{U}})\subseteq G(J)$, it is now enough to show that 
\begin{equation}\label{eq:SR difference}
G(I_{\Delta_{\bal}(U)})\setminus G(I_{\Delta_{\bal}(\widetilde{U})})\subseteq G(J).
\end{equation}
For the left-hand side of \eqref{eq:SR difference} we need to determine those minimal non-faces of $\Delta_{\bal}(U)$ that are faces of $\Delta_{\bal}(\widetilde{U})$. Those are the minimal faces of $F_d(v)=\{j_1^{(i_1)},\ldots ,j_s^{(i_s)}\}\cup\{(m_j+1)^{(j)}~:~j\in [d]\setminus \{i_1,\ldots,i_s\}\}$ not belonging to $\Delta_{\bal}(U)$. Let $F=\{j_1^{(i_1)},\ldots,j_s^{(i_s)}\}$. Then $F\notin \Delta_{\bal}(U)$. Indeed, otherwise, there  exists $\widetilde{u}\in U$ with $F\subseteq F_d(\widetilde{u})$. By definition of $F_d(\cdot)$, it follows that $v$ divides $\widetilde{u}$ and, in particular, as $U$ is an order ideal, we have that $v\in U$, yielding a contradiction. Hence, $F\notin \Delta_{\bal}(U)$. If $G\subseteq F_d(u)$ with $F\not\subseteq G$, then there exists $1\leq \ell \leq s$ such that $j_\ell^{(i_{\ell})}\notin G$. Set $w=\frac{v}{x_{i_\ell,j_\ell}}$. As $U$ is an order ideal, we have $w\in U$ and, by construction, it holds that  $G\subseteq F_d(w)$. Hence, $G\in \Delta_{\bal}(U)$. This shows that $F$ is the only minimal non-face of $\Delta_{\bal}(U)$, belonging to $F_d(v)$. In particular, $G(I_{\Delta_{\bal}(U)})\setminus G(I_{\Delta_{\bal}(\widetilde{U})})=\{v\}$. As $v\notin U$, we have $v\in I(U)$ and hence, $v\in J$. The claim follows.
\end{proof}

We illustrate the previous theorem using the balanced squeezed complex from our running example,  \Cref{RunningExample}.
\begin{example}
The Stanley-Reisner ideal of the complex $\Delta_\bal(U)$ from \Cref{RunningExample} contains the obvious minimal generators $x_{i,\ell}x_{i,3}$ for $1\leq i\leq 3$ and $1\leq \ell\leq 2$ which are forced to lie in $I_{\Delta_\bal(U)}$ because of balancedness. Moreover, $I_{\Delta_\bal(U)}$ contains the squarefree monomials lying in $I(U)$, namely,
\begin{align*}
&x_{1,1}x_{1,2},\quad x_{2,1}x_{2,2},\quad x_{3,1}x_{3,2},\quad x_{1,1}x_{2,1},\quad x_{1,1}x_{2,2},\quad x_{1,1}x_{3,1},\\
& x_{1,1}x_{3,2},\quad x_{1,2}x_{2,1},\quad x_{1,2}x_{3,1},\quad x_{2,1}x_{3,1},\quad x_{2,1}x_{3,2},\quad x_{2,2}x_{3,1},
\end{align*}
where we only list the minimal generators of $I_{\Delta_\bal(U)}$. 
\end{example}

\begin{remark}
It follows from the proof of \Cref{thm:Stanley-Reisner ideal} that given any color-squarefree order ideal $U\subseteq P(d,\mb)$, there exists an ordering $u_1,\ldots,u_m$ of the monomials in $U$ such that $U_\ell=\{u_1,\ldots,u_\ell\}$ for $1\leq \ell\leq m$ is an order ideal. Moreover, $F_d(u_1),\ldots,F_d(u_m)$ is a shelling order of $\Delta_{\bal}(U)$. In particular, if $u_\ell=x_{i_1,j_1}\cdots x_{i_s,j_s}$ with $1\leq i_1<\cdots <i_s\leq d$ and $1\leq j_k\leq m_{i_k}$ for $1\leq k\leq $, then $\{j_1^{(i_1)},\ldots, j_s^{i_s}\}$ is the unique restriction face of $F_d(u_\ell)$. This also implies that
$$
h_S(\Delta_{\bal}(U))=\#\{u\in U~:~\cSupp (u)=S\} \quad \mbox{ for } S\subseteq [d],
$$
which provides another proof of \Cref{cor:h-vector}.
\end{remark}

%%%%%%%%%%%%%%%%%%%%%%%%%%

\section{Multigraded generic initial ideals and Stanley-Reisner ideals of balanced squeezed complexes}
     \label{sec:gins}

The aim of this section is to relate the multigraded generic initial ideal of the Stanley-Reisner ideal of a balanced squeezed complex $\Delta_{\bal}(U)$ to the color-squarefree strongly stable ideal $I(U)$.

We first recall some notions. Let $d\in \N$ and $\mb\in \N^d$. In the following, we consider the polynomial ring $P(d,\mb)$ endowed with an $\N^d$-grading by setting $\deg(x_{i,j})=\eb_i$ for $1\leq i\leq d$ and $1\leq j\leq m_i$. Here, $\eb_i$ denotes the $i$-th unit vector in $\RR^d$. 
The group $G=GL_{m_1}(\K)\times \cdots \times GL_{m_d}(\K)$ acts on $P(d,\mb)$ as the group of $\Z^d$-graded $\K$-algebra automorphisms. More explicitly, the element $\phi=(\phi_1,\ldots,\phi_d)\in G$ with $\phi_j=(a_{k,\ell}^{(j)})_{1\leq k,\ell\leq m_j}$ defines the $\Z^d$-graded automorphism of $P(d,\mb)$ induced by $\phi(x_{j,i})=\sum_{k=1}^{m_j}a_{k,i}^{(j)}x_{j,k}$. Let $\prec$ be a total order on the variables of $P(d,\mb)$, which satisfies 
$$
x_{j,m_j}\prec \cdots \prec x_{j,1}
$$
for all $1\leq j\leq d$. Given a $\Z^d$-graded ideal $I\subseteq P(d,\mb)$  we denote by $\ini_{\prec}(I)$ the initial ideal of $I$ induced by the order $\prec$. It is well-known that there exist Zariski open sets $U_i\subseteq GL_{m_i}(\K)$ for $1\leq i\leq d$ such that $\ini_\prec\phi(I)$ is independent of $\phi$ for all $\phi \in U_1\times \cdots \times U_d$. The ideal $\ini_\prec\phi(I)$ is called \emph{multigraded generic initial ideal} of $I$ with respect to $\prec$. We denote it  by $\gin_\prec(I)$. 

Multigraded generic initial ideals and arbitrary generic initial ideals are known to have similar properties.

\begin{lemma}
Let $\chara(\K)=0$ and let $I\subseteq P(d,\mb)$ be a $\Z^d$-graded ideal. Then the multigraded generic initial ideal $\gin_\prec(I)$ is strongly color-stable.
\end{lemma}

Our main result of this section is the following.

\begin{theorem}\label{thm:gin}
Let $U\subseteq P(d,\mb)$ be a color-squarefree shifted order ideal. Then
$$
I(U)P(d,\mb+(1,\ldots,1))=\gin_\prec(I_{\Delta_{\bal}(U)}).
$$
\end{theorem}

The proof of \Cref{thm:gin} requires several preparations. First, we need to introduce some notation. 
For $1\leq i\leq d$ we set $\mf_i=(x_{i,1},\ldots,x_{i,m_i})$. Let
$$
\Phi:\Mon_{\cs}(d,\mb)\cup \Mon(d,\mb)_2\to \Mon_{\cs}(d,\mb+(1,\ldots,1))
$$
 be the map that is the identity everywhere outside the set $\{x_{i,j}^2~:~1\leq i\leq d,1\leq j\leq m_i\}$ and that maps $x_{i,j}^2$ to $x_{i,j}\cdot x_{i,m_i+1}$ for $1\leq i\leq d$ and $1\leq j\leq m_i$. If $I\subseteq P(d,\mb)$ is a monomial ideal with $G(I)\subseteq \Mon_{\cs}(d,\mb)\cup \Mon(d,\mb)_2$, then we set
$\Phi(I)=\langle \Phi(u)~:~u\in G(I)\rangle$. Note that, by \Cref{thm:Stanley-Reisner ideal}, for a color-squarefree order ideal $U\subseteq P(d,\mb)$ we have 
$$
\Phi(I(U))=I_{\Delta_{\bal}(U)}.
$$
We denote by $\leq_s$ the partial order on $\Mon_{\cs}(d,\mb)$, which is defined as follows: If $u=x_{i_1,j_1}\cdots x_{i_s,j_s}$ and $v=x_{i_1,\ell_1}\cdots x_{i_s,\ell_s}$ are monomials with the same color support (cf., \eqref{eq:multidegree}), then $u\leq_s v$ if and only if $j_k\leq \ell_k$ for all $1\leq k\leq s$. In other words, if $u$ and $v$ have the same color support, then $u\leq_s v$ if and only if for every strongly color-stable ideal $I$ with $v\in I$ one has $u\in I$. We can extend the partial order $\leq_s$ to another partial order $\leq_{cs}$ on $\Mon_{\cs}(d,\mb)$ by setting $u\leq_{cs} v$ if $u$ and $v$ have the same color support  and if $v$ lies in an ideal $I$ that is strongly color-stable across colors, then also $u\in I$. In particular, we have $u\leq_{cs} v$ if $u\leq_s v$. 

The proof of \Cref{thm:gin} will be based on the following two lemmas. 

\begin{lemma}\label{Hilfslemma}
Let $I\subseteq P(d,\mb)$ be a strongly color-stable ideal (resp. strongly colors-stable across colors) with $\mf_1^2+\cdots +\mf_d^2\subseteq I$ that does not contain variables. Let $u=x_{i_1, j_1}\cdots x_{i_s, j_s}\in G(I)$ and let $v\leq_s u$ (resp. $v\leq_{cs} u$). Then $\Phi(v)\in \Phi(I)$. 
\end{lemma}

\begin{proof}
If $v\in \mf_1^2+\cdots +\mf_d^2$, then $v\in G(I)$ and hence $\Phi(v)\in \Phi(I)$. 
Let $v\in \Mon_{\cs}(d,\mb)\setminus (\mf_1^2+\cdots +\mf_d^2)$. By definition of $\leq_s$ (resp. $\leq_{cs}$) also $u$ has to be a color-squarefree monomial. As $I$ is strongly color-stable (resp. strongly color-stable across colors), we conclude that $v\in I$. If $v\in G(I)$, then $\Phi(v)\in \Phi(I)$. If not, then there exists a monomial $w\in G(I)$ that divides $v$. In particular, as $\Phi(w)=w$ divides $\Phi(v)=v$ and $\Phi(w)\in \Phi(I)$, it follows that $\Phi(v)\in \Phi(I)$. This finishes the proof.
\end{proof}

\begin{lemma}\label{lem:degree2}
Let $\mathfrak{m}^2=(x_1,\ldots,x_m)^2\subseteq P(1,m)$. Then, the degree $2$ part of $\mathfrak{m}^2$ equals  the degree $2$ part of $\gin_\prec\Phi(\mathfrak{m}^2)$.
\end{lemma}

\begin{proof}
We note that by definition of $\Phi$, it holds that 
$\dim_\K (\Phi(\mathfrak{m}^2))_2=\binom{m+1}{2}.$ 
Since the Hilbert function is preserved under taking generic initial ideals it follows that 
$$
\dim_\K \left(\gin_{\prec}(\Phi(\mathfrak{m}^2))_2\right)=\binom{m+1}{2}.
$$
On the other hand, as $\Phi(\mathfrak{m}^2)\subseteq P(1,m+1)$ it holds that $\gin_{\prec}(\mathfrak{m}^2)_2\subseteq \Mon(1,m)_2$. As $\# \Mon(1,m)=\binom{m+1}{2}$ we conclude that $$
\gin_{\prec}(\mathfrak{m}^2)_2=\Mon(1,m)_2.
$$
The claim follows. 
\end{proof}

Finally, we can provide the proof of \Cref{thm:gin}:

\begin{proof}[Proof of \Cref{thm:gin}]
%\noindent {\sf Proof of \Cref{thm:gin}:}
By \Cref{rem:h-vector} the Hilbert series of $\K[\Delta_{\bal}(U)]$ is given by
$$
\Hilb(\K[\Delta_{\bal}(U)],t)=\frac{a_0+a_1t+\cdots +a_dt^d}{(1-t)^d},
$$
where $a_i=\#\{u\in U~:~\deg(u)=i\}$. On the other hand,
$$
\Hilb(P(d,\mb)/I(U),t)=a_0+a_1 t+\cdots +a_d t^d.
$$
Hence, if we consider $I(U)$ as ideal in $P(d,\mb+(1,\ldots,1))$, then $\K[\Delta_\bal(U)]$ and $P(d,\mb+(1,\ldots,1))/I(U)$ have the same Hilbert series. It therefore suffices to show that $G(I(U))\subseteq \gin_{\prec}(I_{\Delta_{\bal}(U)})$. 

Let $\phi\in G$ such that $\ini_\prec\phi(I_{\Delta_{\bal}(U)})=\gin_{\prec}(I_{\Delta_{\bal}(U)})$. Let $u=u_1\cdots u_d\in G(I(U))$, with $u_i\in \Mon(1,m_i)$ for $1\leq i\leq d$. We will show the following claim.\\
{\sf Claim:} For $1\leq i\leq d$ there exist monomials $v_{j,1},\ldots,v_{j,k_j}\in \{v\in P(1,m_j)~:~v\leq_s u_j\}$ and $a_{j,1},\ldots,a_{j,k_j}\in \K$ such that 
\begin{equation}\label{eq:claim}
\ini_\prec\phi(a_{j_1}\Phi(v_{j,1})+\cdots +a_{j,k}\Phi(v_{j,k_j}))=u_j.
\end{equation}
The claim then follows in the same way as in the proof of Theorem 1.10 in \cite{Murai-ColorStable}, where instead of \cite[Lemma 1.7]{Murai-ColorStable} we use \Cref{Hilfslemma}. To show the claim it suffices to distinguish two cases. Indeed, by assumption on $U$, we either have that $u$ is color-squarefree, in which case each $u_j$ is either $1$ or equal to a variable $x_{j,\ell}$, or, the monomial $u$ is of the form $x_{j,k}\cdot x_{j,\ell}$ for some $1\leq j\leq d$, $1\leq k,\ell\leq m_j$. 

\emph{Case 1.} $u_j=x_{j,\ell}$ for some $1\leq \ell\leq m_j$. \\
In this situation, the only monomials that are smaller than or equal to $u_j$ w.r.t. the ordering $\leq_s$ are the variables $x_{j,1},\ldots,x_{j,\ell}$ and \eqref{eq:claim} is equivalent to showing that there exist $b_1,\ldots,b_\ell\in \K$ such that 
\begin{equation}\label{eq:case1}
\ini_\prec(b_{1}\phi(x_{j,1})+\cdots +b_{\ell}\phi(x_{j,\ell}))=x_{j,\ell}.
\end{equation}
The initial monomial with respect to revlex order induced by $x_{j,m_j}\prec \cdots x_{j,1}$ of the linear form $b_{1}\phi(x_{j,1})+\cdots +b_{\ell}\phi(x_{j,\ell})$ is of the form $b\cdot x_{j,k}$ where $1\leq k\leq m_j$ is the minimal index such that the coefficient of $x_{j,k}$ in $b_{1}\phi(x_{j,1})+\cdots +b_{\ell}\phi(x_{j,\ell})$ does not vanish. In particular, to show that we can choose $b_1,\ldots,b_\ell\in \K$ such that \eqref{eq:case1} holds we need to show that we can choose $b_1\ldots,b_\ell\in \K$ such that the coefficients of $x_{j,1},\ldots,x_{j,\ell-1}$ do vanish and the one of $x_{j,\ell}$ does not. For the first condition we need to solve a linear system of $\ell -1$ equations, which, as $\phi$ can be chosen to be generic, has a solution. Moreover, since we have one additional degree of freedom, we can guarantee that $x_{j,\ell}$ appears with non-zero coefficient.

\emph{Case 2.} $u=x_{j,k}\cdot x_{j,\ell}$ for some $1\leq j\leq d$, $1\leq k,\ell\leq m_j$. \\
As $U$ is color-squarefree shifted, we have that $\mathfrak{m}_j^2=(x_{j,1},\ldots,x_{j,m_j})^2\subseteq I(U)$. It follows from \Cref{lem:degree2} that there exists a polynomial $f\in \Phi(I(U))$ such that $\ini_{\prec}(f)=u$. This shows the claim.
%\qed
\end{proof}

Once more, we illustrate \Cref{thm:gin} using our running example \Cref{RunningExample}.
\begin{example} 
    \label{exa:gin of running}
We consider the order ideal from \Cref{RunningExample}. We have already seen that the squarefree minimal generators of $I(U)$ are given by:
\begin{align*}
&x_{1,1}x_{1,2},\quad x_{2,1}x_{2,2},\quad x_{3,1}x_{3,2},\quad x_{1,1}x_{2,1},\quad x_{1,1}x_{2,2},\quad x_{1,1}x_{3,1},\\
& x_{1,1}x_{3,2},\quad x_{1,2}x_{2,1},\quad x_{1,2}x_{3,1},\quad x_{2,1}x_{3,1},\quad x_{2,1}x_{3,2},\quad x_{2,2}x_{3,1},
\end{align*}
As $I(U)$ contains all squares of the variables, \Cref{thm:gin} tells us that $\gin_\prec(I_{\Delta_\bal(U)})$ is generated as an ideal in $P(3,(3,3,3))$ by the monomials listed above and the squares $x_{i,j}^2$, where $1\leq i\leq 3$ and $1\leq j\leq 2$.
\end{example}
We conclude this section with a brief comparison of order ideals that are color-squarefree shifted and color-squarefree shifted across colors, respectively. In particular, we study how each property is reflected on the level of the corresponding balanced squeezed complexes, their Stanley-Reisner ideals and their multigraded gins.
The next proposition achieves Fthis for color-squarefree shifted monomial order ideals.

\begin{proposition}\label{prop:comparison}
Let $U\subseteq \Mon_{\cs}(d,\mb)$ be a color-squarefree monomial order ideal with $\mf_1^2+\cdots +\mf_d^2\subseteq I(U)$. Then the following conditions are equivalent:
\begin{itemize}
\item[(i)] $U$ is color-squarefree shifted.
\item[(ii)] $I(U)$ and $\gin_\prec(I_{\Delta_\bal(U)})$ are strongly color-stable.
\item[(iii)] $\Delta_\bal(U)$ is color-shifted.
\item[(iv)] $I_{\Delta_\bal(U)}\cap P(d,\mb)$ is strongly color-stable.
\end{itemize}
\end{proposition}

\begin{proof}
We note that since $I(U)$ and $\gin_\prec(I_{\Delta_\bal(U)})$ have the same set of minimal generators, $I(U)$ is strongly color-stable if and only if $\gin_\prec(I_{\Delta_\bal(U)})$ is.

The equivalence of $(i)$ and $(ii)$ follows from \Cref{lem:order ideal vs ideal} and .

$(i)\Leftrightarrow (iii)$ was shown in \Cref{lem:compareIdealAndComplex}. 

The equivalence $(ii)\Leftrightarrow (iv)$ follows from \Cref{thm:gin}
\end{proof}

Similarly, we have the following result.

\begin{proposition} 
      \label{prop:squeezed across colors}
Let $U\subseteq \Mon_{\cs}(d,\mb)$ be a color-squarefree monomial order ideal. For $1\leq i\leq d$ let $\hat{V}_i=V_i\setminus \{m_i+1\}$ and $\hat{V}=\bigcup_{i=1}^d\hat{V}_i$. Then the following conditions are equivalent:
\begin{itemize}
\item[(i)] $U$ is color-squarefree shifted across colors.
\item[(ii)] $I(U)$ and $\gin_\prec(I_{\Delta_\bal(U)})$ are strongly color-stable across colors.
\item[(iii)] The induced subcomplex $\Delta_\bal(U)_{\hat{V}}$ of $\Delta_\bal(U)$ on vertex set $\hat{U}$ is color-shifted across colors.
\item[(iv)] $I_{\Delta_\bal(U)}\cap P(d,\mb)$ is strongly color-stable across colors.
\end{itemize}
\end{proposition}

\begin{proof}
The equivalence of $(i)$, $(ii)$ and $(iv)$ can be shown in exactly the same way as in the proof of \Cref{prop:comparison}. 
For the equivalence of $(i)$ and $(iii)$ it suffices to note that, since $U$ is an order ideal, the faces of $\Delta_\bal(U)_{\hat{U}}$ are in one-to-one-correspondence with the generators of $U$. 
\end{proof}

%%%%%%%%%%%%%%%%%%%%%%%

\section{Graded Betti Numbers} 
   \label{sec:Betti} 
The focus in this section lies on the multigraded Betti numbers of a particular class of balanced squeezed complexes as well as of more general classes of ideals. 

First, we identify a class of color-square ideals whose graded Betti numbers can be read off directly from their minimal generators. In particular, this class includes all color-squarefree ideals that are strongly color-stable across colors. 

Second, we provide another justification for the term ``balanced squeezed complexes''.  We show that every balanced squeezed complex associated to a color-squarefree order ideal that is shifted across colors has the same graded Betti numbers as the complex obtained by color-shifting it (with respect to a suitable order on the vertices) in the sense of \cite{Babson:Novik}. This is a consequence of a more general result, which says that  any balanced squeezed complex to a color-squarefree shifted order ideal and the multigraded generic initial ideal (with respect to a suitable order) of its Stanley-Reisner ideal have the same $\Z^d$-graded Betti numbers.

Throughout this section we use the following order on variables in $P(d, \mb)$ or $P(d, \mb + (1,\ldots,1))$. We define  $x_{i, j} > x_{k, l}$ if $i < k$ or $i = k$ and $j < l$. 
Observe that this order satisfies the assumption made in \Cref{sec:gins}, and thus the results of this section are applicable. We also remark that this is the reverse order of the order $\preceq$ defined at the beginning of \Cref{sect:prel1}.

Now we consider the following class of color-squarefree monomial ideals.

\begin{definition}
    \label{def:color-sq-free mon ideal}
A monomial ideal $I\subseteq P(d,\mb)$ is called  \emph{color-squarefree} (with respect to the above order of variables) if each of its monomial minimal generators is color-squarefree. The ideal $I$ is said to be \emph{color-squarefree  stable across colors} (with respect to the above order of variables) if it is color-squarefree and  one has: 
\begin{itemize}

\item[(i)]If  $x_{k,\ell}\cdot u\in I$ is a color-squarefree monomial  and $j<\ell$, then $x_{k,j}\cdot u\in I$; \; and 

\item[(ii)] If $x_{k, \ell}$ is the smallest variable dividing a color-squarefree monomial $u \in I$, $x_{i, j} > x_{k, \ell}$ and $x_{i, j} \cdot \frac{u}{x_{k, \ell}}$ is color-squarefree,  then $x_{i, j} \cdot \frac{u}{x_{k, \ell}} \in I$. 
\end{itemize}
\end{definition}

\begin{remark}
(i) Every color-squarefree monomial ideal that is  strongly 
 color-stable across colors is color-squarefree  stable across colors, but the converse is not true. 

(ii) If $m_1 = \cdots = m_d = 1$, then a monomial ideal $I\subseteq P(d,\mb)$ is  color-squarefree  stable across colors if and only if $I$ is squarefree stable  in the standard sense. 
\end{remark}

We want to show that any color-squarefree stable ideal has linear quotients. To this end we need some notation. 

\begin{definition}
    \label{def:smaller var to a monomial}
For a color-squarefree  monomial $u=x_{i_1,j_1}\cdots x_{i_s,j_s} \in P(d, \mb)$ with $i_1 < \cdots < i_s$, let  $\min (u) = x_{i_s, j_s}$ be the smallest variable dividing $u$. Furthermore, set  
\[
\sm (u) = \{x_{p, q} \; : \; p \in [i_s] \setminus \cSupp(u) \text{ and } q \in [m_p]  \text{ or } p = i_k \text{ for some } k \in [s] \text{ and } q \in [j_k-1]\}. 
\]
\end{definition}

\begin{proposition}
    \label{prop:lin quotient}
Consider any color-squarefree stable ideal $I \neq 0$ in $P(d, \mb)$. Among the monomial minimal generators of $I$ with maximum degree let $v$ be the smallest one in the reverse lexicographic order. Let $J \subseteq P(d, \mb)$ be the ideal generated by the monomial minimal generators of $I$ other than $v$. Then $I = (J, v)$, $J$ is color-squarefree stable and 
\begin{align}
    \label{eq:lin quotient}
J : v = (\sm (v)).    
\end{align}
\end{proposition} 

\begin{proof}
Rules (i) and (ii) allow us to replace a color-squarefree monomial $u$ in $I$ by a color-squarefree monomial with the same degree that is larger than $u$. Hence $J$ is color-squarefree stable by the choice of $v$. It remains to establish Equation \eqref{eq:lin quotient}. Write $v=x_{i_1,j_1}\cdots x_{i_s,j_s}$ with $i_1 < \cdots < i_s$. 

We first show the containment $\sm (v) \subseteq J : v$. To this end consider any $x_{k, \ell} \in \sm (v)$. If $k \in [i_s] \setminus \cSupp(u)$, then $k < i_s$, and so $x_{k, \ell} > x_{i_s, j_s}$. Moreover, 
$k \notin \cSupp (v)$ implies that the monomial $v' = x_{k, \ell} \cdot \frac{v}{x_{i_s, j_s}}$ is color-squarefree. It also is greater than $v$, and so it is in $I$ by stability. It follows that $v'$ is in $J$, and 
hence $x_{k, \ell}$ is in $J : v$, as desired. 

If $k = i_p$ for some $p \in [s]$ and $\ell <j_p$, then $x_{k, \ell} > x_{i_p, j_p}$ and $ x_{i_p, \ell} \cdot \frac{v}{x_{i_p, j_p}} > v$ is color-squarefree. Hence, we conclude as above that $x_{k, \ell}$ is in 
$J : v$. 

Now we establish the reverse inclusion $J : v \subseteq (\sm (v))$. Observe that $J : v$ is generated by monomials $\frac{u}{\gcd (u, v)}$, where $u$ is a monomial minimal generator of $J$. Hence it suffices to show: If $\frac{u}{\gcd (u, v)}$ is a minimal generator of $J : v$ for some   monomial minimal generator $u$ of $J$, then $\frac{u}{\gcd (u, v)}$ is in $(\sm (v))$. 

To this end let $u$ be any minimal monomial generator  of $J$ such that $\frac{u}{\gcd (u, v)}$ is a minimal generator of $J : v$. We consider two cases. 
\smallskip 

\emph{Case 1.} Assume there is a variable $x_{k, \ell} > x_{i_s, j_s} = \min (v)$ that divides $u$, but not $v$. Then $x_{k, \ell}$ divides $\frac{u}{\gcd (u, v)}$. If $x_{k, \ell}$ is in $(\sm (v))$, then 
$\frac{u}{\gcd (u, v)}$ is in $(\sm (v))$, as desired. 
Otherwise, $x_{k, \ell} > x_{i_s, j_s}$ forces 
$k = i_p$ for some $p \in [s]$ and $\ell \ge j_p$. Note that $\ell \neq j_p$ because we assumed that 
$x_{k, \ell}$ does not divide $v$. This gives $\ell > j_p$. Hence, by stability, $u' = x_{k, j_p} \cdot \frac{u}{x_{k, \ell}}$ is in $J$ and $\gcd (u', v) = x_{k, j_p} \cdot \gcd (u, v)$, and so 
\[
\frac{u'}{\gcd (u', v)} = \frac{u}{x_{k, \ell} \cdot \gcd (u, v)} \in J : v. 
\]
This shows that $\frac{u}{\gcd (u, v)}$ is not a minimal generator of $J : v$, a contradiction. 
\smallskip 

\emph{Case 2.} Assume that every variable $x_{k, \ell} > x_{i_s, j_s} = \min (v)$ that divides $u$ also divides $v$. Notice that this assumption is also satisfied for $x_{k, \ell} = \min (v)$. 

Suppose first that $\min (u) \ge \min (v)$. Then the assumption for this case gives that $u$ divides $v$, a contradiction to the fact that $v$ is a minimal generator of $I$. Hence it remains to consider the case where $\min (v) > \min (u) = x_{k, \ell}$. By the choice of $v$ with respect to the reverse lexicographic order, this implies $\deg u < \deg v$. Hence, there is a divisor $x_{i_p, j_p}$ of $v$ with $i_p \notin \cSupp (u)$. Thus, $x_{i_p, j_p} \cdot u$  is color-squarefree. 
Since $x_{i_p, j_p} \ge \min (v) > \min (u)$, stability gives $u' = x_{i_p, j_p} \cdot \frac{u}{\min (u)} \in J$.  Observe that $\min (u') > \min (u)$ and every variable that divides $u'$ and is greater than $\min (v)$ divides $v$. If $\min (v)$ is still greater than $\min (u')$, then we repeat the previous step until we get a monomial $\tilde{u} \in J$ with $\min (v) \le \min(\tilde{u})$ and the property that every variable that divides $\tilde{u}$ and is greater than $\min (v)$ divides $v$. As above, it follows that $v$ is not a minimal generator of $I$, a contradiction. 
\end{proof}

Note that $P(d, \mb)$ admits a grading by $\Z^{\mb} = \Z^{m_1} \times \cdots \times \Z^{m_d}$, where the degree of each variable $x_{i, j}$ is the corresponding standard basis vector. For a subset $\sigma$ of variables, we set 
\[
\deg \sigma = \sum_{x_{i, j} \in \sigma} \deg x_{i, j}. 
\]

\begin{corollary}
    \label{cor:Betti numbers color-sqfree stable} 
If  $I \subset P(d, \mb))$ is a color-squarefree stable ideal, then one has for any integer $j \ge 0$, an isomorphism of $\Z^{\mb}$-graded modules 
\[
\Tor^j_{P(d, \mb)} (P(d, \mb)/I, \K) \cong \bigoplus_{u \in G(I)} \bigoplus_{\sigma \subset  \sm(u)} \K(-\deg u - \deg \sigma). 
\]
\end{corollary} 

\begin{proof}
This follows by applying \Cref{prop:lin quotient} and \cite[Lemma 1.5]{HT}. 
\end{proof}

Since the $\Z^{\mb}$-graded Betti numbers determine the $\Z$-graded Betti numbers there is an analogous statement for $\Z$-graded Betti numbers. We leave this to an interested reader. 
\smallskip

Our next goal is to compare the Betti numbers of the Stanley-Reisner ideal to a balanced squeezed  
complex and its multi-graded generic initial ideal. This requires some preparation. 

\begin{lemma}
   \label{lem:an iso}
Consider an ideal $I = (z_1,\ldots,z_m)$ of a polynomial ring $T = \K[z_1,\ldots,z_n, y]$ with $m \le n$. Set $I^{[2]} = (z_i z_j \; : \; 1 \le i < j \le m)$. The $T$-module homomorphism 
\[
R^m \to R/I^{[2]}, \; (f_1,\ldots,f_m)^t \mapsto \left (y \cdot \sum_{j=1}^m z_j f_j \right )\!\!\!\! \mod I^{[2]}
\]
induces an isomorphism of $\Z^{n+1}$-graded $T$-modules 
\[
\bigoplus_{j=1}^m (T/(z_1,\ldots,\hat{z}_j,\ldots,z_m) (-\deg (y z_j)) \to (y I + I^{[2]})/I^{[2]},  
\]
where $\hat{z}_j$ means that the variable $z_j$ is omitted as a generator. 
\end{lemma}

\begin{proof}
Note that the given map induces a well-defined graded homomorphism
\[
\alpha \colon \bigoplus_{j=1}^m (T/(z_1,\ldots,\hat{z}_j,\ldots,z_m) (-\deg (y z_j)) \to T/I^{[2]}
\]
because $z_j \cdot (z_1,\ldots,\hat{z}_j,\ldots,z_m) \subset I^{[2]}$. Its image is $(y I + I^{[2]})/I^{[2]}$. 
Thus, it remains to show that the map $\alpha$ is injective. To this end notice that an element is in the kernel of $\alpha$ if and only if it is in the kernel of the map 
\[
\beta \colon  \bigoplus_{j=1}^m (T/(z_1,\ldots,\hat{z}_j,\ldots,z_m) (-\deg (z_j)) \to T/I^{[2]}
\]
that is induced by $ (f_1,\ldots,f_m)^t \mapsto \left (\sum_{j=1}^m z_j f_j \right )\!\!\!\! \mod I^{[2]}$. 
However, $\beta$ is injective. This follows, for example, from the fact that the domain of $\beta$ and its image, that is, $I/I^{[2]}$ have the same Hilbert series. 
\end{proof}

We return to our standard notation and consider ideals in $P = P(d, \mb)$. Below we will abuse notation by using the same notation for an ideal in a ring $R$ and the ideal it generates in an extension ring of $R$. This should not cause confusion. We set $|\mb| = m_1 + \cdots + m_d$. Recall that $\mm_i$ is the ideal generated by $x_{i,1},\ldots,x_{i, m_i}$. 

\begin{proposition}
    \label{prop:key exact sequence} 
Consider the ideal $I = \mm_1^{[2]} + \cdots + \mm_d^{[2]} \subset P' = P (d, \mb + (1,\ldots,1))$. 
There is an exact sequence of $\Z^{d + |\mb|}$-graded $P'$-modules 
\begin{align}
    \label{eq:basic exact sequence}
0 \to F_d \to \cdots \to F_1 \to F_0 = P'/I \to P'/(I + \sum_{i=1}^d x_{i, m_i + 1} \mm_i), 
\end{align}
where 
\[
F_k = \bigoplus_{u \in [\Mon_{\cs}(P)]_k} (P'/(I : u)) (-\deg (u \prod_{i \in \cSupp (u)} x_{i, m_i+1})). 
\]
\end{proposition}

\begin{proof}
Consider $\mm_i = (x_{i,1},\ldots,x_{i,m_i})$ as an ideal of $R_i = \K[x_{i,1},\ldots,x_{i,m_i+1}]$. Note that $\mm_i^{[2]} : x_{i, j} = (x_{i,1},\ldots,\hat{x}_{i,j},\ldots,x_{i, m_i})$. Hence, for each $i \in [d]$ and $j \in [m_i]$, 
\Cref{lem:an iso} gives an exact sequence of $\Z^{m_i + 1}$-graded $R_i$-modules 
\[
0 \to \bigoplus_{j=1}^{m_i} (R_i/(\mm_i^{[2]} : x_{i, j})) (- \deg (x_{i, j} x_{i, m_i + 1})) \to 
R_i/\mm_i^{[2]} \to R_i/(x_{i, m_i + 1} \mm_i + \mm_i^{[2]}) \to 0. 
\]
Taking the tensor product over $\K$ of the $d$ acyclic complexes 
\[
0 \to \bigoplus_{j=1}^{m_i} (R_i/(\mm_i^{[2]} : x_{i, j})) (- \deg (x_{i, j} x_{i, m_i + 1})) \to 
R_i/\mm_i^{[2]} \to  0
\]
gives an acyclic complex by K\"unneth's formula. It yields the claimed exact sequence using isomorphisms of the form 
\begin{align*}
\hspace{10.5cm}&\hspace{-10.5cm}  
%(R_1/\mm_1^{[2]} : x_{1, j_1}) (- \deg (x_{1, j_1} x_{1, m_1 + 1})) \otimes_K \cdot \otimes_K 
%(R_k/\mm_k^{[2]} : x_{k, j_k}) (- \deg (x_{k, j_k} x_{k, m_k + 1})) \otimes_K 
%R_{k+1}/\mm_{k+1}^{[2]} \otimes_K \cdots \otimes_K R_d/\mm_d^{[2]} \cong 
%
(R_1/\mm_1^{[2]} : x_{1, j_1})  \otimes_\K \cdots \otimes_\K 
(R_k/\mm_k^{[2]} : x_{k, j_k})  \otimes_\K 
R_{k+1}/\mm_{k+1}^{[2]} \otimes_\K \cdots \otimes_\K R_d/\mm_d^{[2]} \\
& \cong P'/(I : x_{1, j_1} \cdots x_{k, j_k}) 
\end{align*}
because 
\[
I : x_{1, j_1} \cdots x_{k, j_k} = \sum_{i=1}^k (\mm_i^{[2]} : x_{i, j_i}) + \sum_{i=k+1}^d \mm_i^{[2]}
\]
as ideals of $P'$. 
\end{proof}

We are ready for the following key result. 

\begin{theorem}
      \label{thm:add polarized squares}
Let $I \subset P = P(d, \mb)$ be a color-squarefree monomial ideal. If 
\[
I' = I + \mm_1^{[2]} + \cdots + \mm_d^{[2]} = I + (x_{i, j} x_{i, k}~:~1 \le i \le d,\ 1 \le j < k \le m_i), 
\]
then one has: 
\begin{itemize}

\item[(a)] There is an exact sequence of of $\Z^{d + |\mb|}$-graded modules over $P' = P (d, \mb + (1,\ldots,1))$ 
\begin{align}
     \label{eq:exact seq}
0 \to G_d \to \cdots \to G_1 \to G_0 = P'/I' \to P'/(I' + \sum_{i=1}^d x_{i, m_i + 1} \mm_i), 
\end{align}
where 
\[
G_k = \bigoplus_{u \in [\Mon_{\cs}(P)]_k} (P'/(I' : u)) (-\deg (u \prod_{i \in \cSupp (u)} x_{i, m_i+1})). 
\]
and the maps are induced by the maps in the exact sequence \eqref{eq:basic exact sequence}. 

\item[(b)] $P'/(I' + \sum_{i=1}^d x_{i, m_i + 1} \mm_i)$ is minimally resolved over $P'$ by iterated mapping cones from \eqref{eq:exact seq}. 

\item[(c)] The $\Z^{d + |\mb|}$-graded Betti numbers of $P'/(I' + \sum_{i=1}^d x_{i, m_i + 1} \mm_i)$  are 
\begin{align*}
\hspace{1.5cm}&\hspace{-1.5cm}  
\beta_{k, a} (P'/(I' + \sum_{i=1}^d x_{i, m_i + 1} \mm_i))  \\
& = \sum_{j=0}^d \left ( \sum_{u \in [\Mon_{\cs}(P)]_k }\beta_{k-j, a - \deg (u \prod_{i \in \cSupp (u)} x_{i, m_i+1})} (P'/(I' : u)) \right ). 
\end{align*}
\end{itemize}
\end{theorem}

\begin{proof}
For every $\mu \in \N_0^{d+|\mb|}$, there is a unique monomial $v \in P'$ of degree $\mu$. Abusing notation we will sometimes identify $v$ with $\mu$. 

Since the  exact sequence \eqref{eq:basic exact sequence} is $\Z_0^{d+|\mb|}$-graded, it gives an 
exact sequence of $\K$-vector spaces in every degree $\mu \in \Z_0^{d+|\mb|}$. Decompose 
\eqref{eq:basic exact sequence} into exact sequences of graded vector spaces by setting 
$F_k = F_k' \oplus F_k''$, where $F_k'$ consists of the graded components in degrees $\mu \in \Z_0^{d+|\mb|}$ with $\mu \notin I'$, and $F_k''$  consists of all other graded components. 

Consider a direct summand of $F_k$, say $(P'/(\mm_1^{[2]} + \cdots + \mm_d^{[2]}) : u) (-\deg (u \prod_{i \in \cSupp (u)} x_{i, m_i+1}))$. This summand gives a non-trivial contribution to $F_k'$ in degree 
$\mu' \in \Z_0^{d+|\mb|}$ if and only if $\mu' = \mu u  \prod_{i \in \cSupp (u)} x_{i, m_i+1} \notin I'$, where $\mu$ is a monomial with 
$[P'/(\mm_1^{[2]} + \cdots + \mm_d^{[2]}) : u]_{\mu} \neq 0$. The latter is equivalent to $\mu \notin I' :  u  \prod_{i \in \cSupp (u)} x_{i, m_i+1} = I' : u$ and $\mu \notin (\mm_1^{[2]} + \cdots + \mm_d^{[2]}) : u$. 
Since $ (\mm_1^{[2]} + \cdots + \mm_d^{[2]}) : u$ is contained in $I' : u$, the last two conditions are equivalent to $[P'/I' :u]_{\mu} \neq 0$. Hence, for every $k \in [d]$, there is an isomorphism of graded vector spaces 
\[
F_k' \to \bigoplus_{u \in [\Mon_{\cs}(P)]_k} (P'/I' : u) (-\deg (u \prod_{i \in \cSupp (u)} x_{i, m_i+1})), 
\]
which is induced by 
\[
\mu u  \prod_{i \in \cSupp (u)} x_{i, m_i+1}\!\!\!\!   \mod ((\mm_1^{[2]} + \cdots + \mm_d^{[2]}) : u) \mapsto 
\mu u  \prod_{i \in \cSupp (u)} x_{i, m_i+1}\!\!\!\!   \mod (I' : u). 
\]
This proves Claim (a). 

For establishing Part (b), observe that, for every monomial $u \in [\Mon_{\cs}(P)]_k$, the ideal $I : u$ 
is generated by color-squarefree monomials that all are in $P$. Therefore, Taylor's resolution implies that 
the $\Z^{d+|\mb|}$-graded Betti numbers of $P'/I' : u$ occur in degrees of squarefree monomials in $P$. 
The direct summands of $G_k$ are of the form $(P'/I' : u) (-\deg (u \prod_{i \in \cSupp (u)} x_{i, m_i+1}))$. It follows that the graded Betti numbers of $G_k$ are supported in degrees of squarefree monomials that are a product of a monomial in $P$ and precisely $k$ of the variables $x_{1, m_1 + 1},\ldots,x_{d, m_d + 1}$. The product of these $k$ variables identifies the Betti number as a contribution of $G_k$. Hence, there can be no cancellations in the mapping cones. 

Claim (c) is a consequence of (b). 
\end{proof}

In the above proof we used arguments for the proof of Theorem 2.1 in \cite{MPS}. In fact, the latter result or a suitable modification of the above proof give the following statement. 

\begin{proposition}
      \label{prop:add  squares}
Let $I \subset P = P(d, \mb)$ be a color-squarefree monomial ideal. If  
\[
I' = I + \mm_1^{[2]} + \cdots + \mm_d^{[2]}, 
\]
then one has: 
\begin{itemize}

\item[(a)] There is an exact sequence of of $\Z^{|\mb|}$-graded $P$-modules 
\begin{align}
     \label{eq:exact seq w squares}
0 \to F_d \to \cdots \to F_1 \to F_0 = P/I' \to P/(I + \sum_{i=1}^d \mm_i^2), 
\end{align}
where 
\[
F_k = \bigoplus_{u \in [\Mon_{\cs}(P)]_k} (P/(I' : u)) (-2 \deg (u)). 
\]
and the maps are induced by the Koszul maps for the regular sequence of $|\mb|$ squares  $x_{1,1}^2, x_{1,2}^2,\ldots,x_{d, m_d}^2$. 

\item[(b)] $P/(I + \sum_{i=1}^d \mm_i^2)$ is minimally resolved over $P$ by iterated mapping cones from \eqref{eq:exact seq w squares}. 

\item[(c)] The $\Z^{|\mb|}$-graded Betti numbers of $P/(I + \sum_{i=1}^d \mm_i^2)$  are  
\begin{align*}
\hspace{1.5cm}&\hspace{-1.5cm}  
\beta_{k, a} (P'/(I' + \sum_{i=1}^d x_{i, m_i + 1} \mm_i))  \\
& = \sum_{j=0}^d \left ( \sum_{u \in [\Mon_{\cs}(P)]_k }\beta_{k-j, a -2  \deg (u )} (P/I' : u) \right ). 
\end{align*}
\end{itemize}
\end{proposition}

\begin{proof}
This follows by \cite[Theorem 2.1]{MPS} and observing that $I' : u = P$ whenever $u \in P$ is a squarefree monomial that is not color-squarefee. In particular, any squarefree monomial $u$ whose total degree is at least $d+1$ is not color-squarefree. 

Alternatively, one can employ the arguments in the proof of \Cref{thm:add polarized squares}. Replace the use of the exact sequence in \Cref{prop:key exact sequence}  by using the Koszul complex on the regular sequence $x_{1,1}^2, x_{1,2}^2,\ldots,x_{d, m_d}^2$. 
\end{proof}

The announced result about comparing Betti numbers follows now quickly. 

\begin{theorem}
     \label{thm:compare Betti numbers}
If $I \subset P = P(d, \mb)$ is a color-squarefree monomial ideal, then the ideals 
\[ 
I + \mm_1^{[2]} + \cdots + \mm_d^{[2]} + \sum_{i=1}^d x_{i, m_i + 1} \mm_i \quad \text{ and } \quad  
I + \mm_1^2 + \cdots + \mm_d^2
\]
of $P' = P(d, \mb + (1,\ldots,1))$  and of $P$, respectively, have the same $\Z^d$-graded Betti numbers. 
\end{theorem}

\begin{proof}
For every  color-squarefree monomial $u \in P$ of total degree $k$, the ideal $I : u$ of $P$ has the same graded Betti numbers as its extension ideal in $P'$. Hence the result follows by comparing Parts (c) of \Cref{thm:add polarized squares} and \Cref{prop:key exact sequence} and observing that the monomials $u^2$ and $u \prod_{i \in \cSupp (u)} x_{i, m_i+1}$ have the same $\Z^d$-degree. 
\end{proof}

As a consequence of the last result we obtain the following statement for any balanced squeezed complex. 

\begin{corollary}
    \label{cor:compare Betti of balanced squeezed} 
Let $U \subset P(d, \mb)$ be a color-squarefree shifted order ideal and consider the associated balanced squeezed complex $\Delta_{\bal}(U)$. Its Stanley-Reisner ideal $I_{\Delta_{\bal}(U)}$ in $P(d, \mb + (1,\ldots,1))$ and its multigraded generic initial ideal have the same $\Z^d$-graded Betti numbers. 
\end{corollary}

\begin{proof}
By \Cref{lem:order ideal vs ideal}, the ideal $I(U)$ of $P = P(d, \mb)$ can be written as 
\[
I(U) = J + \mm_1^2 + \cdots + \mm_d^2, 
\]
where $J$ is a color-squarefree monomial ideal. 

According to \Cref{thm:Stanley-Reisner ideal}, the Stanley-Reisner ideal of $\Delta_{\bal}(U)$ is an ideal of $P' = P(d, \mb + (1,\ldots,1))$, namely 
\[
I_{\Delta_{\bal}(U)} = J +  \mm_1^{[2]} + \cdots + \mm_d^{[2]} + \sum_{i=1}^d x_{i, m_i + 1} \mm_i.  
\]
Since, by \Cref{thm:gin}, one has 
\[
\gin_\prec(I_{\Delta_{\bal}(U)}) = I(U) \cdot P'
\]
we conclude using \Cref{thm:compare Betti numbers}. 
\end{proof}

\begin{example}
We consider once more the order ideal from \Cref{RunningExample}. The $\Z$-graded Betti numbers of both, $I_{\Delta_\bal(U)}$ and $\gin_\prec(I_{\Delta_\bal(U)})$,  are given by 
\begin{table}[h!]
    \[\begin{array}{rccccccc} 
			& 0 &1 & 2&3 &4&5&6  \\
			\text{total:} & 1 & 18 & 56 & 79 & 60 & 24 & 4\\%\hline
       0: &1 &\cdot & \cdot &\cdot &\cdot &\cdot &\cdot \\ %\hline
			1: &\cdot &18 & 53& 69 &48 &18 &3 \\ %\hline
			2: &\cdot &\cdot & 3 &9 &9 &3 &\cdot \\ %\hline
			3: &\cdot &\cdot & \cdot &1 &3 &3 &1 \\ %\hline
    \end{array}
\]
\end{table}
\end{example}
Finally, we relate our construction to algebraic color-shifting as introduced by Babson and Novik in \cite{Babson:Novik}. Let $\Gamma$ be a balanced simplicial complex. Pass to the generic initial ideal $I$ of the Stanley-Reisner ideal of $\Gamma$. Applying then a colored ``polarization'' map to $I$ gives a squarefree monomial ideal in a suitable polynomial ring, which is, by definition, the Stanley-Reisner ideal of the  complex $\widetilde{\Gamma}$ obtained by color-shifting from $\Gamma$ (see \cite{Babson:Novik} for details). 

\begin{example}
Consider the balanced squeezed complex $\Delta_{\bal}(U)$ of our running example (see \Cref{RunningExample}). The multigraded generic initial ideal of its Stanley-Reisner ideal is described in \Cref{exa:gin of running}. Applying the colored polarization map to its minimal generators, one computes that the Stanley-Reisner ideal of the complex $\widetilde{\Delta_{\bal}(U)}$ obtained by color-shifting is generated by the color-squarefree minimal generators of $I(U)$, that is, by 
\begin{align*}
& x_{1,1}x_{2,1},\quad x_{1,1}x_{2,2},\quad x_{1,1}x_{3,1},\quad x_{1,1}x_{3,2},\\ 
&  x_{1,2}x_{2,1},\quad x_{1,2}x_{3,1},\quad x_{2,1}x_{3,1},\quad x_{2,1}x_{3,2},\quad x_{2,2}x_{3,1}, 
\end{align*}
by $x_{1,1}x_{1,3},\ x_{2,1}x_{2,3},\ x_{3,1}x_{3,3}$ and by the monomials $x_{i,j} x_{i,j+1}$ with $1\leq i\leq 3$ and $1\leq j\leq 2$. 
\end{example}

\begin{theorem}
     \label{thm:compare with shifted} 
Let $U \subset P(d, \mb)$ be a color-squarefree  shifted order ideal. Then the balanced squeezed complex $\Delta_{\bal}(U)$ and the complex $\tilde{\Delta}$ obtained by color-shifting it have the same $\Z$-graded Betti numbers. 
\end{theorem}

\begin{proof}
By \Cref{cor:compare Betti of balanced squeezed}, the ideals $I_{\Delta_{\bal}(U)}$ and 
$\gin_\prec(I_{\Delta_{\bal}(U)})$ have the same $\Z^d$- and thus the same $\Z$-graded Betti numbers. Since $U$ is color-squarefree shifted by assumption, the ideal $I(U)$ and so 
$\gin_\prec(I_{\Delta_{\bal}(U)})$ are strongly color-stable monomial ideals. Thus the main result, Theorem 0.1,  in \cite{Murai-ColorStable} gives that $\gin_\prec(I_{\Delta_{\bal}(U)})$ and the Stanley-Reisner ideal of $\tilde{\Delta}$ have the same $\Z$-graded Betti numbers. 
\end{proof}

%\begin{example}
%The color-shifted complex of the balanced squeezed complex from \Cref{RunningExample} has the following set of facets
%\begin{align*}

%\end{align*}
%\end{example}
%\todo[inline]{Uwe: remains TO DO: add running example, incorporate references below} 

%%%%%%%%%%%%%%%%%%%%%%

\bibliographystyle{alpha}
\bibliography{biblio}

\end{document}